 \documentclass{article}
\usepackage{amsmath}
\usepackage{color}
\usepackage{amsfonts} 
\usepackage{amssymb}
\usepackage{amsthm}   
\usepackage{latexsym}
\usepackage{fancyhdr}
\usepackage[left=2cm, right=1.5cm, top=2cm, bottom=2cm]{geometry}

\numberwithin{equation}{section}  
\newtheorem{theorem}{Theorem}[section]
\newtheorem{definition}[theorem]{Definition}    

\newtheorem{lemma}[theorem]{Lemma}
\newtheorem{remark}[theorem]{Remark}

\newtheorem{example}[theorem]{Example}

\newcounter{exocpt}

\renewcommand{\H}{{\mathcal H}}
\def\C{\mathbb C}
\def\R{\mathbb R}
\def\S{{\mathcal S} }

\def\C{\mathbb C}
\def\R{\mathbb R}

\def\N{\mathbb N}
\def\al{\alpha}

\def\be{\beta}

\def\de{\delta}
\def\rh{\rho}
\def\et{\eta}
\def\th{\theta}

\def\GA{\Gamma}

\def\ve{\varepsilon}

\def\la{\lambda}

\def\om{\omega}
\def\va{\varphi}

\def\va{\varphi}

\def\g{\mathfrak{g}}

\def\h{\mathfrak{h}}

\def\t{\mathfrak{t}}
\def\z{\mathfrak{z}}

\def\g{\mathfrak g}
\def\h{\mathfrak h}

\def\ot{\otimes}
\def\la{\lambda}
\def\ve{\varepsilon}
\def\si{\sigma}
\def\om{\omega}

\def\ga{\gamma}
\def\Ga{\Gamma}
\def\ph{\phi}
\def\ch{\chi}

\def\ps{\psi}
\def\si{\sigma}
\def\N{\mathbb{N}}
\def\Z{\mathbb{Z}}
\def\R{\mathbb{R}}
\def\C{\mathbb{C}}
\def\HH{\mathbb{H}}

\def\ol#1{\overline{#1}}
\def\nn{\nonumber}
\def\noop#1{\Vert #1\Vert_{\rm op}}
\def\R{{\mathbb R}}
\def\C{{\mathbb C}}
\def\N{{\mathbb N}}

\def\Z{{\mathbb Z}}
\def\T{{\mathbb T}}

  \def\Id{{\mathbb I}}
\def\A{{\mathcal A}}
\def\B{{\mathcal B}}
\def\D{{\mathcal D}}
\def\F{{\mathcal F}}

\def\H{{\mathcal H}}

\def\K{{\mathcal K}}

\def\ZZ{{\mathcal Z}}

\def\O{{\mathcal O}}

\def\iy{\infty}

\def\ol#1{\overline{#1}}

\def\hb#1{\hbox{#1}}
\def\val#1{\vert #1\vert}

\def\no#1{\Vert #1\Vert }

\def\wh#1{\widehat{#1}}

\def\exp#1{{\rm exp} #1}
\def\ind#1#2{\hb{ind}_{#1}^{#2}}
\def\CC#1{ C_c(#1)}

\def\limk{\lim_{k\to\infty}}
\def\res#1{_{\vert #1}}

\def\hb#1{\hbox{#1}}
\def\val#1{\vert #1\vert}

\def\L1#1{L^1(#1)}

\def\L#1#2{L^{#1}(#2)}
\def\l#1#2{L^{#1}(#2)}
\def\Im{\mathrm{\, Im \,}}
\def\Re{\mathrm{\, Re\, }}
\def\ti{\times }

\def\lef({\left(}
\def\rig){\right)}
\def\lan{\langle}
\def\ran{\rangle}

\begin{document}
%opening
\title {Tensor products of $\textit{NCDL}-C^*$-algebras and the  $C^*$-algebra of the Heisenberg motion groups $\T^n\ltimes\HH_n$.  }

\author{Hedi REGEIBA and Jean LUDWIG}

\date{}

\maketitle
%\maketitle

\begin{abstract}
We show that the tensor product $A\otimes B$ over $\C$  of two $C^* $-algebras satisfying the \textit{NCDL} conditions has again the same property. 
 We use  this result to describe the $C^* $-algebra of the Heisenberg  motion groups $G_n=\T^n\ltimes\HH_n $ as algebra of operator fields defined over the spectrum of $G_n $.

\end{abstract}

\section{Introduction.}\label{intro}

\subsection{}
The family of $C^*$-algebras with  norm controlled dual limits (NCDL) was introduced in \cite{Lud-Reg1}. 
It is shown  in section \ref{tensor prod}, 
that the  tensor product of two \textit{NCDL}-$C^*$-algebras  is again \textit{NCDL}. 
In section \ref{CstarGone}, we introduce the groups $G_n=\T^n\ltimes \HH_n , n\in\N^*$,  
the semi-direct product  of the torus $\T^n $ acting on the $2n+1 $  dimensional Heisenberg group $\HH_n $.  We recall the topology of the spectrum of the groups $G_n $ and the Fourier transform  of their group $C^*$-algebras. In section \ref{ncdl of Gone},   the norm control of dual limits is then computed  explicitly for the group $G_1 $. This is the main result of the paper. In the last section,  the structure of the $C^* $-algebra of $G_n$ is obtained by combining  the general  
  results on tensor products of $C^* $-algebras and the properties of the algebra of operator fields $\F(C^*(G_1)) $.
  
\subsection{$C^* $-algebras with norm controlled dual limits}.\begin{definition}\label{lcd} 
\rm
\begin{itemize}\label{}$ $

\item  Let $S$ be  a topological space. 
We say that $S $ is \textit{locally compact of step $\leq d$}, if  
there exists a 
finite increasing family $ \emptyset \ne  S_0\subset S_{1}\subset\cdots\subset S_d=S $ of closed 
subsets of $S $,  such that   the subsets $\Gamma_0=S_0$ and  
$ \Gamma_i:=S_{i}\setminus S_{i-1}$,  $i=1,\dots, d$,   
are  locally compact and Hausdorff in  their relative topologies.

\item   Let $S$ be locally compact of step $\leq d$, and let $\{\H_i\}_{i=1, \dots, d}$ be Hilbert spaces. 
For a closed subset $ M \subset S $, denote by $ CB(M,\H_i) $ the unital $C^*$-algebra of all uniformly 
bounded operator fields  $ (\psi(\gamma)\in \B(\H_i))_{\gamma\in M\cap \Gamma_i, i=1,\dots, d}$, 
which are 
operator norm continuous on the subsets $ \Gamma_i \cap M$ for every $  i\in\{0,\dots, d\} $ with  $ \Gamma_i\cap M\ne\emptyset $ and such that $\ga\mapsto \psi(\ga) $ goes to 0 in operator norm if $\ga $ goes to infinity  on $M$.   
We provide the algebra $ CB(M,\H_i) $ with the infinity-norm
$$
\no{\varphi}_{M}=\sup\left\{\no{\varphi(\gamma)}_{\B(\H_i)}\mid M \cap \Gamma_i\ne \emptyset, \, \gamma\in M \cap \Gamma_i\right\}.
$$
\item  Let $S $ be a set. Choose  for every $s\in S $ a Hilbert space $\H_s $. 
We define the $C^* $-algebra $l^\iy(S) $ of uniformly bounded operator fields  defined over $S $ by
\begin{eqnarray*}
 l^\iy(S):=\{(\ph(s))_{s\in S}\vert \ \ph(s)\in  \B(\H_s), s\in S, \ \sup_{s\in S}\noop{\ph(s)}<\iy\}.
 \end{eqnarray*}
Here $\B(\H) $ denotes  the algebra of bounded linear operators on the Hilbert space $\H $. 
 \end{itemize}
\end{definition}

\begin{definition}\label{norcontspec}\rm
Let $\A $ be a separable liminary $ C^* $-algebra, such that   
  the spectrum $\widehat{\A} $ of $\A$ is a locally compact space of step $\leq d$, 
 $$ \emptyset= S_{-1}\subset  S_0\subset S_{1}\subset \cdots \subset S_d=\widehat{\A}. $$  
 Suppose that  for  $0\le i \leq d $ 
there  is a Hilbert space $\H_i $, and for every $ \gamma\in \Gamma_i $  a concrete realization $ (\pi_\gamma,\H_i) $ of 
$\gamma $ on the Hilbert space $ \H_i $ and that the  set $ S_0 $ is the collection of all 
characters of $\A$. 

Denote by $\F:\A\to l^\iy(\wh \A) $ the Fourier transform of $\A $ i.e. for $a\in \A $ let
\begin{eqnarray*}
 \F(a)(\ga)=\hat a(\ga):=\pi_\ga(a)\in\B(\H_i), \ga\in \Ga_i, i=0,\cdots, d.  
 \end{eqnarray*}

We say that $\F(\A )$ is \textit{continuous of step} $\leq d $  
 $\F(A)\res{\GA_i} $ is contained in $CB(\hat A,\H_i) $ for every $i $.

\end{definition}
\begin{definition}\label{norconduli}
\rm  Let $\A $ be a separable liminary $ C^* $-algebra.

We say that the $C^* $-algebra $\A$ has  \textit{norm controlled dual limits (NCDL)}  if the spectrum $\hat \A $ is continuous of step $\leq d $ for some $d\in\N $ ($\emptyset= S_{-1}\subset  S_0\subset S_{1}\subset \cdots \subset S_d=\widehat{\A}  $) and 
\begin{enumerate}
\item  $\F(\A) $ is continuous of step $\leq d $.
\item\label{nrocontspec_2}    For any  $ i=1,\dots, d$  and for any 
converging sequence contained in $ \Gamma_i $ with limit set contained  in $S_{i-1}$, 
there exists a properly converging sub-sequence $\overline{\gamma}=(\gamma_k)_{k\in\N} $ with limit set  $L(\ol\ga)(\subset \S_{i-1}) $,   a constant  $ C>0 $
and for every $ k\in\N $  an involutive  linear mapping $ \sigma_{\overline{\gamma},k}: \F(A)\res{L(\ol\ga)}\to \B(\H_i)$, 
that is  bounded by $ C$, such that
$$
(\forall a\in\A)\quad \lim_{k\to\infty}\no{\F (a)(\gamma_k)-
\sigma_{\overline\gamma,k} (\F (a)\vert_{L(\ol\ga)})}_{\B(\H_i)}=0.
$$
 \end{enumerate}
 \end{definition}
 
\begin{remark}\label{conti finite step implies ncdl}
\rm   It turns out that the norm control of dual limits is  a consequence  of  the properties of  liminary $C^* $-algebras with continuous Fourier transform of finite step:
 \end{remark}

\begin{theorem}[see \cite{Be-Be-Lu}]\label{newcor}
Assume that for the separable liminary  $C^*$-algebra $\A$,  its Fourier transform $\F(\A) $ 
is continuous of step $\leq d $ for some $d\in \N $. Then $\A$ has norm controlled dual limits. 
Specifically
let 
$1\leq \ell \leq d$ be fixed, and  
$\bar{\gamma}=(\gamma_k)_{k} \in \Gamma_\ell$ be a properly convergent  sequence with limit set $L(\ol\ga) $ outside $\Gamma_\ell$. 
Then there exists a sequence   $(\sigma_{\bar{\gamma}, k})_k$ 
of completely positive and completely contractive maps 
$\sigma_{\bar{\gamma}, k}\colon \hat{\A}\res{L(\ol\ga)}  
\to \B(\H_\ell)$ 
such that 
$$(\forall a\in\A)\quad  
\lim_{k\to \infty} \no{\F(a) (\gamma_k)- \sigma_{\bar{\gamma}, k}(\F(a)\vert_{\hat{\A}\res{L(\ol\ga)} })}_{\B(\H_{\ell})}=0.$$
\end{theorem}

\begin{remark}\label{importanve NCDL}
\rm   Let us mention that any $NCDL -C^*$-algebra $A $ is determined up to an isomorphism by its spectrum $\hat A $ 
and the family $\si_{\ol\ga}\ ( \ol\ga \text{ any properly converging sequence in }\hat A) $ of norm controls (see \cite{Be-Be-Lu} and \cite{Lud-Reg1}). Therefore in order to determine the structure of a given $NCDL-C^* $-algebra and to understand its Fourier transform,  it is essential to have a precise description of its spectrum and of its  norm control of dual limits.

Let $A $ be a separable $C^* $-algebra and let $\ol\ga=(\ga_k)_k\in\N $ be a properly converging sequence of irreducible unitary representations of $A $. Let $L $ be the set of limits in $\hat A $ of the sequence $\ol\ga $. By definition of the topology of $\hat A $, (see \cite{Dixmier}), there exists for every $\si \in L $ and every element $\xi_\si$ in the Hilbert space $\H_\si  $ of $\si $ a sequence $(\xi_k\in \H_{\ga_k})_k $, such that the sequence of coefficients $c^{\ga_k}_{\xi_k} $ converges weakly to the coefficient $c^{\si}_{\xi_\si} $. This means that 
\begin{eqnarray*} 
 \langle \si(a)\xi_\si,\xi_\si\rangle_{\H_\si} &= & 
 \limk \langle \ga_k(a)\xi_k,\xi_k\rangle_{\H_{\ga_k}},\ a\in A. 
\end{eqnarray*}
By  a theorem of Fell (\cite{Fell-1}) we have that
\begin{eqnarray*} 
 \limk \noop{\pi_k(a)} &= &\sup_{\si\in L}\noop{\si(a)}, a\in A. 
\end{eqnarray*}
In this paper these sequences   $(\xi_k)_k $ are  explicitly determined for generic sequences $\ol{\ga_k} \subset \hat G_1 $ and it is shown how to  construct the control from these data (see the proof of theorem \ref{estimates of sequences rne 0}). 

 \end{remark}

 \section{Tensor Products of $C^*$-algebras and \textit{NCDL} $C^*$-algebras.}\label{tensor prod}×
 
\subsection{Tensor products}
(see \cite{Bal} for details) \ 
If $A$ and $B$ are $C^*$-algebras, we can  form their algebraic tensor product $A\otimes B$ over $\C$. The vector space 
 $A\otimes B$ has a natural structure 
 as a $*$-algebra 
 with multiplication
 $$(a_1\otimes b_1)(a_2\otimes b_2)=a_1a_2\otimes b_1b_2$$
 and involution $(a\otimes b)^*=a^*\otimes b^*$. 
 If $\ga$ is a $C^*$-norm on $A\otimes B$, we will write $A\otimes_\ga B$ for the completion.
 If $\pi_A$ and $\pi_B$ are representations of $A$ and $B$ on 
 Hilbert spaces $\H_1$ and $\H_2$ respectively, we can form the representation  $\pi=\pi_A\otimes\pi_B$ of $A\otimes B$ on $\H_1\otimes\H_2$ by
 $\pi(a\otimes b)=\pi_A(a)\otimes\pi_B(b)$. For any $\pi_A$ and $\pi_B$ we have 
 $\no{(\pi_A\otimes\pi_B)(\sum_{i=1}^na_i\otimes b_i)}\leq\sum_{i=1}^n\no{a_i}\no{b_i}$, so the norm 
 $\Vert\sum_{i=1}^na_i\otimes b_i\Vert_{min}=\sup\Vert(\pi_A\otimes\pi_B)(\sum_{i=1}^na_i\otimes b_i)\Vert$ 
 is finite and hence a $C^*$-norm called the 
 minimal $C^*$-norm. The completion of $A\otimes B$ with respect to this norm is written $A\otimes_{min}B$ and called the minimal or spatial tensor product
 of $A$ and $B$. 
 For any representation $\pi$ of $A\otimes B$ we have $\no{\pi(\sum_{i=1}^na_i\otimes b_i)}\leq\sum_{i=1}^n\no{a_i}\no{b_i}$, so the 
 norm $\no{\sum_{i=1}^na_i\otimes b_i}_{max}=\sup\no{\pi(\sum_{i=1}^na_i\otimes b_i)}$ is finite and hence a $C^*$-norm called the maximal $C^*$-norm.
 The completion is denoted $A\otimes_{max}B$, and called the maximal tensor product of $A$ and $B$. 
 
 Let $A$ a $C^*$-algebra, $A$ is called nuclear if for every $C^*$-algebra $B$, there is a unique $C^*$-norm on $A\otimes B$, i.e $A\otimes_{min}B=A\otimes_{max}B$.   We write for this nuclear $C^* $-algebra $A $ and any other $C^* $-algebra $B $
 \begin{eqnarray*}
 C=A\ol \ot B
 \end{eqnarray*}
for the $C^* $-algebra $A\ot_{\max}B=A\ot_{\min }B $. 

Every type $I $ $C^* $-algebra is nuclear. 

 \begin{theorem}(\cite{Bal} IV.3.4.21)
  If $A$ and $B$ are arbitrary $C^*$-algebras, there is an injective map 
  $$\varPi:\wh A\ti\wh B\to\wh{A\otimes B}$$
  given by $\varPi(\rh,\si)=\rh\otimes\si.$ This is a continuous map relative to the natural topologies and drops to a well-defined map, also denoted $\varPi$, 
  from $Prim(A)\ti Prim(B)$ to $Prim(A\otimes B)$; this $\varPi$ is injective, continuous (it preserves containment in the appropriate sense), and its range 
  is dense in $Prim(A\otimes B)$ since the intersection of the kernels of the representations $\{\rh\otimes\si:\ \rh\in\wh A,\ \si\in\wh B\}$ is $0$.
 \end{theorem}
\begin{theorem}\label{thbijwhAtiwhBandwhAotiB}(\cite{Bal} IV.3.4.27)
 If $\pi\in\wh{A\otimes B}$ and $\pi_A$ or $\pi_B$ is of type I, the other is also type $I$ and $\pi\cong\rh\otimes\si$ for $\rh$, $\si$ 
 irreducible representations quasi-equivalent to $\pi_A$ and $\pi_B$ respectively. Thus, if $A$ or $B$ is type $I$ (no separability necessary), the map 
 $\varPi:\wh A\ti\wh B\to\wh{A\otimes B}$ is surjective and it is easily verified to be a homeomorphism.
\end{theorem}
\begin{theorem}\label{tensor ncdl}
 Let $A$ and $B$ two \textit{NCDL}-$C^*$-algebras. Then the tensor product $A\otimes B$ over $\C$ is also a 
 \textit{NCDL}-$C^*$-algebra. 
\end{theorem}
\begin{proof}
 Let $A$,  $B$ be two \textit{NCDL}-$C^*$-algebras. Since both algebras  are by definition liminary, they are both of type I. Furthermore $C $ is separable, since so are $A $ and $B $. 
 By theorem \ref{thbijwhAtiwhBandwhAotiB} we have that $\wh C\simeq\wh A\ti\wh B$ and so $C $ is liminary too. 
We can write the spectrum of $C$ in the following way.
There are  increasing finite families $S_0^A\subset S_1^A\ldots\subset S_{n}^A=\wh A,\ n\in\N$, resp. $S_0^B\subset S_1^B\ldots\subset S_{m}^B=\wh B,\ m\in\N$, 
of closed subsets  of the spectrum $\wh A$ of $A$, resp. of the spectrum $\wh B$ of  $B$, 
such that the subsets $\GA^{A}_j:=S_j\setminus S_{j-1}, j=1,\cdots, n,  $ resp. $\GA^{B}_j:=S^{B}_j\setminus S^{B}_{j-1}, j=1,\cdots, m$ have  a separated  relative topology.
Then the spectrum of $C$  is the disjoint union of the subsets
$$\GA^{C}_{i,j} =\GA^{A}_i\times \GA^{B}_j, 0\leq i\leq n, 0\leq j\leq m,$$
and each subset $\GA^C_{i,j} $ is locally compact with a Hausdorff topology. 
Let for $k\in\{0,\cdots,n+m \} $
\begin{eqnarray*}
 T^{C}_k:=\bigcup_{i+j\leq k}\GA^{C}_{i,j}.
 \end{eqnarray*}
Then for every $0\leq i\leq n,0\leq j\leq m
, $ the subset 
\begin{eqnarray*}
 R^{C}_{i,j}:=\Ga^{C}_{i,j}\cup T^{C}_{i+j-1}
 \end{eqnarray*}
is closed in $\wh C $. For every $0\leq k\leq n+m $ we choose a total order on  the family 
\begin{eqnarray*}
R^{C}_k :=\{\GA^{C}_{i,j}\vert i+j=k\}
 \end{eqnarray*}
 and we say that $\GA^{C}_{i,j}\subset\GA_{i',j'}^{C} $ whenever $i+j<i'+j' $. This gives us  a total order on the family of sets $\{\GA^{C}_{i,j}\vert \ 0\leq i\leq n,0\leq j\leq m\}.  $ 
Furthermore 
\begin{eqnarray*}
 \wh A\times\wh B =\bigcup_{i,j}\GA^{C}_{i,j}
 \end{eqnarray*}
and the subsets 
\begin{eqnarray*}
 S^{C}_{i,j}:=\bigcup _{(i',j')\leq (i,j)}\GA^{C}_{i',j'}
 \end{eqnarray*}
are closed in $\wh C $ for any pair $(i,j) $. 

It is easy to see that $C$ has all the required properties to be \textit{NCDL}. Indeed for 
$c\in C$ it is  immediately seen that the  operator fields $\wh c$ defined on $\wh C$ 
by $\wh c(\rh\otimes\si)=\rh\otimes \si(c)$, $\rh\in\wh A,\ 
\si\in\wh B$ operate on the Hilbert-spaces $\H_i\otimes\H_j $ and it is continuous on the different subsets $\GA_{i,j}^C$ and tends to $0$ at infinity, since this is true for elementary tensors $a\otimes b $. 

Hence  the Fourier transform of $A\otimes B $ is continuous of some finite step and so Theorem \ref{newcor} tells us that $C $ has the NCDL property. 

Let us see how to build the norm control  of  dual limits for $C=A\ol\otimes B $.
 \begin{enumerate}\label{}
\item If $\ol{(\ga^C)}=(\ga_k^C)_k$ is a sequence in $\GA_{i,j}^C$ which admits its limits in 
$T^{C}_{i+j-2}$, then for a properly convergent subsequence we have 
$$\ga_k^C=\ga_k^A\otimes\ga_k^B\in \wh C,$$
where $(\ga_k^A,\ga_k^B)$ is a sequence of $\GA_i^A\ti\GA_j^B$ for some $0\leq i\leq n$ and $0\leq j\leq m$, which converges to its limit set 
$L\left(\ol{(\ga^A)}=(\ga_k^A)_k\right)\ti L\left(\ol{(\ga^B)}=(\ga_k^B)_k\right)$ in $S^{A}_{i-1}\times S^{B}_{j-1} \subset \wh A\ti\wh B$.
Since $A $ and $B $ are nuclear, we have that
\begin{eqnarray*}
 \F(C)\res{L(\ga_k^C)}=\F(A)\res{L(\ol{\ga^A})}\ol\ot \F(B)\res{L(\ol{\ga^B})}.
 \end{eqnarray*}

Let $(\si_{\ol{(\ga^A)},k}:\F (A)\res{L(\ol{\ga^A)}}\to B(\H_i))_k,\ (\text{resp.}\  
((\si_{\ol{(\ga^B)},k}:\F (B)\res{L(\ol{\ga^B)}}\to B(\H_j))_k)$ be the 
sequence of uniformly bounded linear mappings that comes  from  the  \textit{NCDL} property for $A$ 
(resp. for $B$). Define then for all $k\in\N$:
$$ \si_{\ol{\ga^C},k}(\phi^A\otimes\phi^B)=\si_{\ol{\ga^A},k}(\phi^A)\otimes\si_{\ol{\ga^B},k}(\phi^B)$$
on elementary tensors. This definition can be extended in a unique way to a bounded selfadjoint linear mapping
\begin{eqnarray*}
 \si_{\ol{\ga^C},k}: \F(C)\res{L(\ol{\ga^C})}\to \B(\H_i\otimes \H_j),k\in\N.
 \end{eqnarray*}
 
Then we see that for all finite sums $c=\sum_l a_l\otimes b_l\in C$: 
\begin{eqnarray*}
 \underset{k\to\iy}{\lim}\noop{\wh c(\ga_k^C)-  
 \si_{\ol{\ga^C},k}({\wh c}\res{L(\ol{\ga^C})}}
 &=&
 \underset{k\to\iy}{\lim}\noop
 {
 \sum_l \wh a_l(\ga_k^{A})\otimes \wh b_l(\ga_k^{B})-
\si_{\ol{\ga^A},k}({\wh a_l}
  \vert_
 {L(\ol{\ga^A})})\otimes \si_{\ol{\ga^B},k}
 (
 {\wh b_l}
 \vert_{L(\ol{\ga^B})
 }
 )
 }
 \\
 \nn  &=
 &
0.
\end{eqnarray*}
Furthermore
\begin{eqnarray*}
 \noop{\si_{\ol{\ga^C},k}(c)}\nn  &= &
 \noop{\sum_ l\si_{\ol{\ga^A},k}({\wh a_l}
  \vert_
 {\ol{\ga^A}})\otimes \si_{\ol{\ga^A},k}
 (
 {\wh b_l}
 \vert_{
 \ol{\ga^B}
 })}\\
 \nn  &\leq &
\sum_ l\noop{\si_{\ol{\ga^A},k}({\wh a_l}
  \vert_
 {\ol{\ga^A}})}\noop{ \si_{\ol{\ga^B},k}
 (
 {\wh b_l}
 \vert_{
 \ol{\ga^B}
 })}\\
 \nn  &\leq &
\be_{\ol{\ga^A}}\be_{\ol{\ga^B}} \sum_l\no{{\wh a_l}
  \vert_
 {\ol{\ga^A}}}_\iy \no{{\wh b_l}
  \vert_
 {\ol{\ga^B}}}_\iy,
 \end{eqnarray*}
 for some constants $\be_{\ol{\ga^A}}>0$ and $\be_{\ol{\ga^B}}>0.$

Hence
\begin{eqnarray*}
 \noop{\si_{\ol{\ga^C},k}(c)}\nn  &\leq &
 \be_{\ol{\ga^A}}\be_{\ol{\ga^B}} \no{\F(c)\res{L(\ol{\ga^C})}}.
 \end{eqnarray*}
 \item  
 Similarly, if $\ol{(\ga^C)}=(\ga_k^C)_k$ is a sequence in $\GA_{i,j}^C$ which admits its limits in $S^{A}_{i-1}\otimes \GA^{B}_j $
then for a properly convergent subsequence we have 
$$\ga_k^C=\ga_k^A\otimes\ga_k^B\in \wh C,$$
where $(\ga_k^A\otimes\ga_k^B)_k$ is a sequence of $\GA_i^A\ti\GA_j^B$ for some $0\leq i\leq n$ and $0\leq j\leq m$ which converges to its limit set 
$L\left(\ol{(\ga^A)}=(\ga_k^A)_k\right)\ti \{\ga^{B}\}$ in $S^{A}_{i-1}\times \GA^{B}_j \subset \wh A\ti\wh B$ for some $\ga^{B} \in \GA^{B}_j$.
  
Let $(\si_{\ol{(\ga^A)}}:CB(S_{i-1}^A)\to B(\H_i))_k,$ be the 
sequence of uniformly bounded linear mappings that comes  from  the  \textit{NCDL} property for $A$. Define then for all $k\in\N$:
$$\ga_k^C: CB(T_{i+j-1}^C)\to B(\H_i\ti \H_j):\ \ga_k^C(\phi^A\otimes\phi^B)=\ga_k^A(\phi^A)\otimes \phi^B(\ga^{B}).$$
Then we see again for all finite sums $c=\sum_l a_l\otimes b_l\in C$ that: 
\begin{eqnarray*}
 \underset{k\to\iy}{\lim}\noop{\wh c(\ga_k^C)-
 \tilde\si_{\ol{\ga^C},k}({\wh c}\res{\GA_{r-1}^C})}
 &=&
 \underset{k\to\iy}{\lim}\noop
 {
 \sum_l \wh a_l(\ga_k^{A})\otimes \wh b_l(\ga_k^{B})-
 \tilde\si_{\ol{\ga^A},k}({\wh a_l}
  \vert_
 {\ol{\ga^A}})\otimes 
{\wh b_l}
 (\ga^{B})
 }\\
 \nn  &=
 &
0.
\end{eqnarray*}
 \end{enumerate}

\end{proof}
\begin{example}\rm
[The $ C^* $-algebra of a direct product of a locally compact group second countable  and a locally compact abelian second countable group]

Let $G$ a second countable  locally compact group and $A$ a locally compact second countable abelian group. 

Let $\tilde G:=G\ti A$ be  the direct product of $G$ and $A$. Then  
 $$\wh{\tilde G}=\wh G\ti\wh A.$$
For $F\in C_c(\tilde G)$ we define the application 
 \begin{eqnarray}\label{defwhFA}
 \wh F^A: \wh A\to L^1(G);\ \wh F^A(\chi)(g):=\int_A F(g,a)\chi(a)da,\ g\in G,\chi\in\wh A.
 \end{eqnarray}
Then  $\wh F^A$   is a continuous mapping which 
extends for all $b\in C^*(\tilde G)$ into a continuous mapping  $$\wh b^F:\wh A\mapsto C^*(G).$$
If now  $C^*(G)$ is \textit{NCDL}, we can write the spectrum of  $C^*(\tilde G)$ the following way:\\
There is an   increasing finite family $S_0\subset S_1\subset\cdots\subset S_d=\wh G$ of  closed subsets of the spectrum
$\wh G$ of $G$ such that  
for all  $i=1,\cdots,d,$ the subset  $\GA_0=S_0$ and  $\GA_i:=S_i\backslash S_{i-1},\ i=1,\cdots,d,$ 
are separated for their relative topology. 
Then the spectrum of  $\tilde G$ is the disjoint union of the subsets  $\tilde S_j:=S_j\ti\wh A,\ j=0,\cdots,d$
and for all  $j=1,\cdots,d,$ the subset  
$\tilde\GA_j:=\tilde S_j\backslash\tilde S_{j-1}$ has Hausdorff relative topology.\\
It is easy to see that $C^*(\tilde G)$ has all the required properties to be \textit{NCDL}. Indeed for $F\in C_c(\tilde G)$ it is immediate to see 
 that the operator fields $\wh F$ defined on   $\wh{\tilde G}$ 
by $\wh F(\pi\ti\chi)=(\pi\ti\chi)(F)=\pi(\wh F^A(\chi)),\ \pi\in\wh G,\ 
\chi\in\wh A,$ are continuous on the different subsets $\tilde\GA_j$ and they go to $0$ at infinity.\\

\end{example}

\section{The $C^*$-algebra of the Heisenberg motion groups $G_n$.}\label{CstarGone}
The structure of the group $C^*$-algebra $C^*(G) $ realized as algebra of operator fields defined over the spectrum $\wh G $ of $G $ is already known for certain classes of Lie groups, such as the Heidelberg and the thread-like Lie groups
(see \cite{Lud-Tur}) and the $ax+b$-like groups (see \cite{Lin-Lud}). Furthermore, the $C^*$-algebras of the $5$-dimensional nilpotent Lie 
groups have been determined in \cite{Lud-Reg1}, while those of all $6$-dimensional nilpotent Lie groups have been characterized in \cite{Reg}
and the $C^*$-algebras of the two-step nilpotent Lie groups have been determined  in \cite{Gun-Lud}. Furthermore it follows from general principles that the $C^* $-algebra of any connected nilpotent Lie group has the \textit{NCDL}-property (see \cite{Be-Be-Lu}).
The description  of the $C^*$-algebra of the motion group $SO(n)\ltimes\R^n$ has been done  in  \cite{Lud-Ell-Abd}. 
\subsection{The Heisenberg motion groups $G_n$.}
We denote by $\text{diag}(\ga_1,\cdots,\ga_n)$ a diagonal matrix in $\text{Mat}(n,\C)$ with numbers $\ga_1,\cdots,\ga_n.$
Let $\HH_n=\C^n\times\R$ denote the $(2n+1)$-dimensional Heisenberg group, with group law 
$$(z,t)(z',t')=\left(z+z',t+t'-\frac{1}{2}\text{Im}(z\cdot\bar{z'})\right),\ z,z'\in\C^n,\ t,t'\in\R,$$
where $\text{Im}(z)$ is the imaginary part of $z$ in $\C^n$ and $z\cdot z':=\overset{n}{\underset{j=1}{\sum}}z_jw_j.$

The group $\T^n$ acts naturally on $\HH_n$ by automorphisms as follows 
$$e^{i\theta}(z,t):=(e^{i\theta}z,t),$$
where $e^{i\theta}=(e^{i\theta_1},\cdots,e^{i\theta_n})\in \T^n$.

Let $G_n$ be the semi-direct product $\T^n\ltimes\HH_n$, 
equipped with the following group law:
 \begin{eqnarray*}
 (e^{i\theta},z,t)\centerdot(e^{i\theta'},z',t'):
 =\left(e^{i(\theta+\theta')},z+e^{i\theta}z',t+t'-\frac{1}{2}\text{Im}(z\cdot\overline{e^{i\theta}z'})\right) ,
 \forall e^{i\theta},e^{i\theta'}\in\T^n\text{ and } (z,t),(z',t')\in\HH_n.
 \end{eqnarray*}
 For $z\in\C^n,$ we introduce the $\R$-linear form $z^*$ on $\C^n$ defined by 
 $$z^*(w):=\text{Im}(z\cdot\overline{w})$$
and we identify the algebraic dual of  the Lie algebra $\t_n $ of $\T^n \subset U(n)$ with $i \R^n$ via the scalar product
\begin{eqnarray*}
 iA\cdot i B=-\sum A_j B_j,\ iA,i B\in\t_n.
 \end{eqnarray*}

 We have a map 
 \begin{eqnarray*}
  \left.\begin{array}{cccc} 
 & & & \\
\times: & \C^n\ti \C^n& \longrightarrow& \t^*_n\\
&(z,w)&\mapsto&z\ti w\end{array}\right.
 \end{eqnarray*}
given by 
$$z\ti w(A):=w^*(Az)=\text{Im}(w\cdot\overline{Az})=\sum_{j=1}^n\text{Re}(w_j\ol{z_j})A_j,\ A=(iA_1,\cdots, iA_n)\in \t_n.$$
It follows easily from the group law in $G_n$ that the coadjoint representation $\text{Ad}^*$ of $G_n$ is given by 
\begin{eqnarray*}
 \text{Ad}^*(e^{i\theta},z,t)(U,u,x)=\left(U+z\ti (e^{i\theta}u)+\frac{x}{2}z\ti z,e^{i\theta}u+xz,x\right),
\end{eqnarray*}
for all $(U,u,x)\in\g_n.$ Therefore the coadjoint orbit of $G_n$ through $(U,u,x)$ is given by 
\begin{eqnarray*}
 \O_{(U,u,x)}&=&\text{Ad}^*(G_n)(U,u,x)\\
 &=&\left\{(U+z\ti (e^{i\theta}u)+\frac{x}{2}z\ti z,e^{i\theta}u+xz,x);\ e^{i\theta}\in\T^n,\ z\in\C^n\right\}.
\end{eqnarray*}

\begin{definition}\label{ir def}
\rm   Let for $0\ne r=(r_1,\cdots, r_n)\in \R^n_+ $ 
\begin{eqnarray*}
 I_r:=\{j\in \{1,\cdots, n\}\vert r_j\ne 0\}
 \end{eqnarray*}
and 
\begin{eqnarray*}
 \Z^n_r:=\{(\la_1,\cdots, \la_n)\in \Z^n, \la_j=0,\  \forall j\in I_r\}.
 \end{eqnarray*}
Let also
\begin{eqnarray*}
 \T^n_r:=\{e^{i\th}\in \T^n\vert \th=(\th_1,\cdots, \th_n), \th_j\in\Z,\ \forall j\in I_r\}.
 \end{eqnarray*}
Then $\T^n_r $ is  a closed connected subgroup $\T^n $ isomorphic to $\T^{n-\val {I_r}} $.
 \end{definition}
Furthermore, the subgroup $\T^n_r $ is the stabilizer group of the linear functional 
$\ell_r=(ir_1,\cdots, ir_n)\in \t_n^* $ and $\Z^n_r $ describes the spectrum of the group $\T^n_r $.

It follows   that  the space $\g_n^\ddag/G_n$ of admissible coadjoint orbits of 
$G_n$ is the union of the set $\GA^2 $ of all orbits
\begin{eqnarray*}
 \O_{(\la,\al)}=\left\{(i(\la_1-\frac{\al}{2}(x_1)^2),
 \cdots,i(\la_n-\frac{\al}{2}(x_n)^2)),\al e^{i\theta}x,\al);\ 
 e^{i\theta}\in\T^n,\ x=(x_1,\cdots,x_n)\in\R_+^n\right\}, 
\end{eqnarray*}
$\text{ where  } \la=(\la_1,\cdots,\la_n)\in \Z^n \text{ and }\al\in\R^*$, of the set $\GA^1 $ of  all orbits 
\begin{eqnarray*}
 \O_{\la,r}=\left\{i(\la+z\ti(e^{i\theta}r)),e^{i\theta}r,0);\ e^{i\theta}\in\T^n,\ z\in\C^n\right\}, r\in \R^n_+,r\ne 0, \la\in \Z^n_r.
\end{eqnarray*}
(this corrects  a mistake in   \cite{Hal-Rah},) and of the set $\GA^0 $ of all  the one point orbits 
\begin{eqnarray*} 
 \O_\la &= &\{\ell_\la=(i\la,0,0)\}, \la\in \Z^n. 
\end{eqnarray*}

In the case of $\GA^1 $, we parametrize its  orbits  by 
$$\GA^1=\{\ell_{\la,r}=(i\la,r,0);\ r\in\R_+^n,r\ne 0, \ \la\in \Z^n_r\}.$$

We can thus write  the space $\g_n^\ddag/G_n$ of admissible coadjoint orbits of $G_n$ as the disjoint  union 
\begin{eqnarray*}
 \g_n^\ddag/G_n&=&\GA^2\cup\GA^1\cup \GA^0\\
 \nn  &= &
(i\Z^n\ti\R^*)\bigcup _{r\in\R^n_+,\ r\ne 0,\ \la\in \Z^n_r}(i\la,r)\bigcup_{\la\in\R} \ell_\la.
 \end{eqnarray*}

 The topology of the space $\hat G_1 $ has been determined in \cite{th-Ell} and   of the space $\hat G_n $   in \cite{Hal-Rah} (at least partially). 
We have the following description of the topology of  $\g_n^\ddag/G_n $.
\begin{theorem}\label{limitset}$ $
\begin{enumerate}
 \item Let $(\O_{(\la^k,\al_k)})_k$ be a sequence of admissible coadjoint orbits of $G_n.$ Then   
 \begin{enumerate}
 \item $(\O_{(\la^k,\al_k)})_k$ converges to 
 $\O_{\la,\al}$ in $\g_n^\ddag/G_n$ if and only if  the sequence $(\al_k)_k$ converge to $\al$ 
 and $\la^k=\la$ for $k $ large enough.
 \item $(\O_{(\la^k,\al_k)})_k$ converges to 
 $\O_{\la,r} $ in $\g_n^\ddag/G_n$ if and only if  the sequence $(\al_k)_k$ converge to zero,
 for all $j\in I_r,\ \al_k\la_j^k$ tends to $\frac{r^2_j}{2}$ and for all $j\not\in I_r $, $\al_k\la_j^k$ 
 tends to $0$ and $\al_k(\la_j^k-\la_j)\geq 0 $ as  $k\to+\iy$.
 
 \end{enumerate}
 \item Let $(\O_{\la_k,r^k})_k$ be a sequence of admissible coadjoint orbits of $G_n$ such that $I_{r_k}=I $ is constant. Then 
 $(\O_{\la_k,r^k})_k$ converges to $\O_{\la,r}\in\GA^1$  if and only 
 if $\underset{k\to\iy}{\lim}r^k=r$ and $\la^k_j=\la_j $ for all $j\not\in I_r$.
\item Let $(\O_{\la_k,r^k})_k$ be a sequence of admissible coadjoint orbits of $G_n$ such that $I_{r_k}=I $ is constant. Then 
 $(\O_{\la_k,r^k})_k$ converges to $\O_{\la}\in\GA_0$  if and only 
 if $\underset{k\to\iy}{\lim}r^k=0$ and $\la^k_j=\la_j $ for all $j\not\in I_r$.

\end{enumerate}
\end{theorem}

  \begin{proof}$ $
 \begin{enumerate}
  \item 
  \begin{enumerate}
\item See Theorem $5$ in \cite{Hal-Rah}.
\item
  \begin{enumerate}
   \item If $I_r=\{1,\cdots,n\}$, see  Theorem $4$ in \cite{Hal-Rah}.
   \item If $I_r=\{\varnothing\}$, see Theorem $3$ in \cite{Hal-Rah}.
   \item If $I_r\subsetneq\{1,\cdots,n\}$, assume that $(\O_{\la^k,\al_k})_k$ converges to $\O_{\la,r}.$ Then there exist two 
   sequences $(e^{i\th^k})_k\subset\T^n$ and $(x^k)_k\subset\R_+^n$ such that 
   \begin{eqnarray}
    \left.\begin{array}{cccc}
 \al_k & \longrightarrow&0& \\
\al_ke^{i\th_j^k}x^k_j& \longrightarrow& r_j,&\forall\ j\in I_r\\
\al_ke^{i\th_j^k}x^k_j& \longrightarrow& 0,& \forall\ j\notin I_r\\
\la_j^k-\frac{\al_k}{2}(x^k_j)^2& \longrightarrow& 0,& \forall\ j\in I_r\\
\la_j^k-\frac{\al_k}{2}(x^k_j)^2& \longrightarrow& \la_j,& \forall\ j\notin I_r
\end{array}\right.
   \end{eqnarray}
We have $|\al_k|x_j^k\longrightarrow r_j$ for all $j\in I_r$ and $|\al_k|x_j^k\longrightarrow 0$ for all 
$j\notin I_r$. Since $\al_k(\la_j^k-\frac{\al_k}{2}(x^k_j)^2)\longrightarrow0$ for all $j\in\{1,\cdots,n\},$ we immediately
see that the sequence $(\al_k\la_j^k)_k$ tends to $\frac{r_j^2}{2}$ for all $j\in I_r$ 
and $(\al_k\la_j^k)_k$ tends to zero for all $j\notin I_r$. We also  have 
$\al_k(\la_j^k-\la_j)\geq0$ for large $k$ for all  $j\notin I_r.$ 

Conversely, let us assume that $(\al_k)_k$ converges to zero, for all $j\in I_r$, $(\al_k\la_j^k)_k$ 
converges to $\frac{r_j^2}{2}$ and for all $j\notin I_r,\ (\al_k\la_j^k)_k$ converges to zero and $\al_k(\la_j^k-\la_j)\geq0$. 
Then for $k$ large enough we can define for all $j\in I_r$ the sequence $x_j^k=\sqrt{\frac{2\la_j^k}{\al_k}}.$ We  see that 
for all $j\in I_r$ $\al_ke^{i\th_j^k}x_j\longrightarrow r_j.$ For all $j\notin I_r$ we have assumed that the sequence 
$(\al_k\la_j^k)_j$ converges to zero and $\al_k(\la_j^k-\la_j)\geq0$ for large $k$. Take for all $j\notin I_r$ 
$x_j^k=\sqrt{\frac{2}{\al_k}(\la_j^k-\la_j)}$.

This  the sequence $(\O_{(\la^k,\al_k)})_k$ converges to $\O_{\la,r}$ in $\g_n^\ddag/G_n.$
   \end{enumerate}
\end{enumerate}
\item The orbits $\O_{\la_k,r_k} $ are direct products of orbits of $\T\ltimes \HH_1 $. It suffices to apply  \cite{E.L.} or Theorem $6$ in \cite{Hal-Rah}.

\end{enumerate}

 \end{proof}

 As a consequence we obtain (see \cite{E.L.}, \cite{Lud-Ell-Abd} and  \cite{Hal-Rah})
\begin{theorem}\label{homeomorphism dual orbit space}
 The unitary dual $\widehat G_n$ is homeomorphic to the space of admissible coadjoint orbits $\g_n^\ddag/G_n.$  
\end{theorem}

\subsection{The Fourier Transform.}
According to Theorem \ref{homeomorphism dual orbit space} the spectrum $\hat G_n $ is determined by the space of admissible coadjoint orbits. Hence we have irreducible representations of the form $\pi_{\la,\al},(\al\in\R^*, \la\in \Z^n $), $\pi_{\la,r}, (r\in \R_{>0},\la\in \Z^n_r $) and characters $\ch_\la,\la\in \Z^n $.

\subsubsection{The generic representations $\pi_{\la,\al} $}\label{generic} They are extensions of the infinite dimensional irreducible representations $\pi_\al,\al\in\R^*, $ to $G_n $.

The irreducible representation $\pi_\al  $ of $\HH_n $ acting on the space $\F_\al(n) $ for $\al>0 $ is given by Folland in \cite{Fo}.  
  Let for $u,v\in \C^n $
\begin{eqnarray*}
uv=u\cdot v:= (\sum_{j=1}^n u_j v_j),\  u=(u_1,\cdots, u_n),\ v=(v_1,\cdots,  v_n).
 \end{eqnarray*}

For $\al>0 $, the  Hilbert space of $  \pi_\al $ is the space 
\begin{eqnarray*} 
 \F_\al(n)  &:= &\left\{f:\C^n\to \C; f\text{ holomorphic }, \no f_\al^2=\al\int_{\C^n}\val{f(w)}^2e^{-\frac{\al}{2} \val w^2}dw<\iy\right\}.  
\end{eqnarray*}
by, taking for  the character $ \ch_\al$  the expression
\begin{eqnarray*}
 \ch_\al(t)=e^{i\al t},\ t\in\R. 
 \end{eqnarray*}
We have 
\begin{eqnarray*} 
 \pi_\al(z,t)\xi(w) &= &e^{i\al t}e^{-\frac{\al}{4}\vert  z\vert ^2}e^{-\frac{\al}{2} w\ol z}\xi(w+z). 
\end{eqnarray*}
On the other hand, if $\al < 0 $, the Fock space $\F_\al(n) $ consists of antiholomorphic functions $f:\C^n\to \C $ such that
\begin{eqnarray*} 
 \no f_\al^2 &:= &\val \al\int_{\C^n}\val{f(w)}^2e^{-\frac{\al}{2}\val{w}^2}dw <\iy.
\end{eqnarray*}
The representation $\pi_\al $ takes the form
\begin{eqnarray*} 
 \pi_\al(z,t)f(w) &= & e^{i \al t}e^{-\frac{\al}{4}\vert  z\vert ^2}e^{-\frac{\al}{2} \ol w z}\xi(w+\ol z). 
\end{eqnarray*}
Therefore the representation $\pi_{\la,\al}  $ acts for $\al>0  $ on $\F_\al(n)
$ by
\begin{eqnarray*} 
 \pi_{\la,\al}(e^{i\th},z,t)f(w) &= &e^{i\la \th}e^{i \al t}e^{-\frac{\al}{4}\vert  z\vert ^2}e^{-\frac{\al}{2} w\ol z}\xi(e^{-i\th}w+e^{-i\th}z),
\end{eqnarray*}
and for $\al<0 $:
\begin{eqnarray*} 
 \pi_{\la,\al}(e^{i\th},z,t)f(w) &= &e^{i\la \th}e^{i \al t}e^{-\frac{\al}{4}\vert  z\vert ^2}e^{-\frac{\al}{2} \ol w z}\xi(e^{i\th}w+e^{i\th}\ol z). 
\end{eqnarray*}
For $F\in L^1(G_n) $ , $f\in \F_{\al}(n), \al>0 $ we then  have that
\begin{eqnarray}\label{pilaal F}
 \pi_{\la,\al}(F)f(w)&=&
\int_{\T^n}\int_{\HH_n}e^{i\la \th}e^{i \al t}e^{-\frac{\al}{2}\vert  z\vert ^2}e^{-\al /2 w\ol z}F((e^{i\th},z,t))f(e^{-i\th}w+e^{-i\th}z) dzd\th dt\\
\nn  &= &
\int_{\T^n}\int_{\C^n}e^{i\la \th}e^{-\frac{\al}{4}\vert  z-w\vert ^2}e^{-\al/2 w\ol{ (z-w)}}\widehat F^3{(e^{i\th},z-w,\al)}f((e^{-i\th }z) dzd\th\\
\nn  &= &
\nn\int_{\T^n}\int_{\C^n}e^{i\la \th}e^{-\frac{\al}{4}\vert  z\vert ^{2} +
\frac{\al}{4}\vert w\vert ^2-\frac{i\al}{2} \Im( w\ol z)}\widehat F^3{(e^{i\th},z-w,\al)}f((e^{-i\th }z) dzd\th.
 \end{eqnarray}

\subsubsection{The infinite dimensional representations $\pi_{\la,r} $ (which are trivial on the centre of $G_n $).}\label{pilar}
Let $\ch_{\la,r}(0\ne r=(r_1,\cdots, r_n)\in \R_+^n ,\la\in\Z^n_r) $ be the unitary character 
\begin{eqnarray*}
 \ch_{\la,r}((e^{i\th},z,t)):=e^{i\th\cdot \la }e^{i \text{Re} (\ol z\cdot r)},e^{i\th}\in\T^n_r, z\in\C^n, t\in\R.
 \end{eqnarray*}

The representation $\pi_{\la,r}=\ind {\T^n_r\times\HH_n}{G_n}{\ch_{\la,r} }$ acts then on the space 
\begin{eqnarray*}
 \H_{\la,r}:=L^2(\T^n/\T^n_r,\ch_{\la,r}) 
 \end{eqnarray*} 
by
\begin{eqnarray*}
 \pi_{\la,r}(e^{i\th},z,t)\xi(e^{i\mu})=e^{i\th\cdot \la}e^{-i \text{Re}(e^{i(\th-\mu)} \ol z\cdot r)}\xi(e^{i(\mu-\th)}).
 \end{eqnarray*}

Hence, for $F\in L^1(G_n),\ \xi\in L^2(\T^n/\T^n_r,\la),\  e^{i\mu}\in \T^n/\T^n_r $: 
\begin{eqnarray*}
 \pi_{\la,r}(F)\xi(e^{i\mu})\nn  &= &
\int_{\T^n}\int_{\C^n}e^{i\th\cdot \la}e^{-i \text{Re}(e^{i\mu} \ol z\cdot r)}\widehat F^3(e^{i\th},z,0)\xi((e^{i(\mu-\th)})d\th dz\\
\nn  &= &
\int_{\T^n/\T^n_r}\widehat F^{1,2,3}(e^{i(\mu-\th)},\la,e^{i\mu}\cdot \ell_r,0)\xi(e^{i\th })d\th,
 \end{eqnarray*}
where 
\begin{eqnarray*}
 \widehat F^{1,2,3}(e^{i\th},\la,r,0):=\int_{\T^n_r}e^{i\la \va} e^{-i\Re(r\cdot \ol z)}F(e^{i(\th+\va)},z,t)d\va dzdt.
 \end{eqnarray*} 
 
{
%\color{red}
\subsubsection{The characters.}\label{charac}

Let for $ \la\in\Z^n$ 
\begin{eqnarray*}
 \ch_\la(e^{i\th},z,t):=e^{i\la\cdot \th}, \th\in\R^n,z\in\C^n,\ t\in\R.
 \end{eqnarray*}
Then the set $ \{\ch_\la,\ \la\in\Z\}$  is the collection of all unitary characters of the group $ G_n$. 

In particular for the group $G_1 $ we obtain the partition of $\wh G_1 $ into  three Hausdorff subsets
\begin{eqnarray*}
 \wh G_1\nn  &= &\GA_2:=\{\pi_{\la,\al}\vert \la\in\Z,\al\in\R^*\}\cup
\GA_1:=\{\pi_{\la,r}\vert r>0,\la\in\Z\}\cup
\GA_0:=
\{\ch_\la\vert \la\in\Z\}.
 \end{eqnarray*}
  \section{The \textit{NCDL} property for $C^*(G_1) $. }\label{ncdl of Gone}

\subsection{Some definitions}

\begin{enumerate}\label{basis and bessel}
\item 
Let 
\begin{eqnarray*} 
 b_{N,\al}(z) &:= &\sqrt{\frac{\al^{N}}{2^N N!}}z^N\\
 \nn  &= &
b_{N,1}(\sqrt \al z), z\in \C, 
\end{eqnarray*}
be the $N $'th orthonormal  vector of the canonical Hilbert basis of $\F_\al(1) $.
\item  
We define for $ N\in\N, M\in \R_+, N\geq M$  the orthogonal projection $ P_N$ of $ \l2\T$  by
\begin{eqnarray*}
 P_N\left(\sum_{j\in\Z}c_j \ch_j\right):=\sum_{j\geq-N}c_j \ch_j.
 \end{eqnarray*}
 and
 \begin{eqnarray*}
 P_{N,M}\left(\sum_{j\in\Z}c_j \ch_j\right):=\sum_{\underset{j\geq-N}{\val j\leq M} }c_j \ch_j.
 \end{eqnarray*}
Let also 
 \begin{eqnarray*}
 \l2\T_{N}:=P_{N}(\l2\T)=\left\{\sum_{{j\geq-N} }c_j\ch_j;\ \sum_{j\in\Z}\val{c_j}^2<\iy\right\},\\
 \l2\T_{N,M}:=P_{N,M}(\l2\T)=\left\{\sum_{\underset{j\geq-N}{\val j\leq M} }c_j\ch_j;\ \sum_{j\in\Z}\val{c_j}^2<\iy\right\}.
 \end{eqnarray*}

Define for $ \et=\underset{j\in\Z}{\sum} c_j \ch_j\in\l2\T$ 
the element
\begin{eqnarray*}
 V_k(\et):=\sum_{j=-\la_k}^\iy i^{j}c_j b_{j+\la_k,\al_k}
 \end{eqnarray*}
of the Fock space $ \F_\al(1)$. We see that the mapping $ \et\to V_k(\et)$   from  the space
\begin{eqnarray*}
 \l2 \T_{\la_k}:=\left\{\sum_{j=-\la_k}^\iy c_j \ch_j;\ \sum_{j}\val{c_j}^2<\iy \right\}
 \end{eqnarray*}
onto the Fock space $ \F_{\al_k}(1)$ is linear, isometric and surjective. We have that
\begin{eqnarray*}
 V_k^*(f)=\sum_{j=0}^{\iy} (-i)^{j-\la_k}\ga_j \ch_{j-\la_k};\ f_k=\sum_{j=0}^{\iy}\ga_j b_{j,\al_k}. 
 \end{eqnarray*}
Then
\begin{eqnarray}\label{identity Vk}
 V_k\circ V_k^*=\Id_{\F_{\al_k}(1)}, k\in\N.
 \end{eqnarray}
\item  We shall make an essential use of Bessel functions in the proof Theorem \ref{estimates of sequences rne 0} (see \cite{watson} for the definition and properties of Bessel functions).
\begin{definition}\label{bessel}
\rm   Let for $ n\in\Z$  and $ z\in \C$ 
\begin{eqnarray*}
 J_n(z):=(\frac{z}{2})^n\sum_{k+n\geq 0,k\geq 0}^{\iy}(-1)^k\frac{(\frac{1}{4}z^2)^k}{k!(k+n)!}.
 \end{eqnarray*}
Then,
\begin{eqnarray*}
 J_n(z)=\frac{i^{-n}}{\pi}\int_{0}^{\pi}e^{iz\cos(\th)}\cos(n\th)d\th, z\in \C,n\in\Z.
 \end{eqnarray*}
Write for $ z\in \C$, 
\begin{eqnarray*}
  z=e^{i\nu}\val z.
 \end{eqnarray*}
Then for $ v\in\C$, 
\begin{eqnarray*}
 J_n(-\val v)
 \nn  &= &
\frac{i^{-n}}{\pi}\int_{0}^{\pi}e^{-i\vert v\vert \cos(\th)}\frac{e^{in\th}+e^{-in\th}}{2}d\th\\
\nn  &= &
\frac{i^{-n}}{2\pi}\int_{-\pi}^{\pi}e^{-i\Re(\vert v\vert e^{i\th)}}e^{-in\th}d\th.
 \end{eqnarray*}
 \end{definition} 
 \end{enumerate}

 \subsection{Convergence to $\pi_{r} $}  
 \begin{lemma}\label{pik et de}
 For $ F\in\l1{G_1}$ and $ l, j\in\Z$  we have that 
 \begin{eqnarray*} 
 \nn  & &
\langle\pi_{\la_k,\al_k}(F)V_k(\ch_j),V_k(\ch_l)\rangle
 \nn \\
\nn  &=&
\sum_{\overset{l-j+q\geq 0}{0\leq q\leq j+\la_k}}
(-1)^{q+l-j} i^{l-j}\frac{{\sqrt 
{
{(l+\la_k)!}
{(j+\la_k)!}
}}}{q!(j+\la_k-q)!}\frac{(\la_k\al_k/2 )^{q+(l-j)/2}}{\la_k^{q+(l-j)/2}(q+l-j)!}\\
\nn  & &
 \int_{\C}(\ol z)^{l-j} \vert  z\vert ^{2q}e^{-\frac{\al_k}{4}\vert  z\vert ^2}\hat F^{1,3}((-j,z,\al_k))dz\\ 
 \nn  &\underset{(\text{if }l\geq j)}{=} &
\sum_{\overset{l-j+q\geq 0}{0\leq q\leq j+\la_k}}
(-1)^{q+l-j} i^{l-j}\frac{{\sqrt {(\la_k+j+1)\cdots (\la_k+l) }(j+\la_k-q+1 )\cdots (j+\la_k))
}}{\la_k^{q+(l-j)/2}(q!)^2}\frac{(\la_k\al_k/2 )^{q+(l-j)/2}}{(q+1)\cdots(q+l-j)}\\
\nn  & &
\int_{\C}(\frac{\ol z}{\val z})^{(l-j)} \vert  z\vert ^{2q+l-j}e^{-\frac{\al_k}{4}\vert  z\vert ^2}\widehat F^{1,3}((-j,z,\al_k))dz
\end{eqnarray*}
and for $l,j\in\Z $
\begin{eqnarray*} 
  \langle \pi_{r}(F)(\ch_j),\ch_l\rangle 
  \nn  &= &
\sum_{\overset{l-j+q\geq 0}{0\leq q\leq +\iy}}  \int_{\C}
(-1)^{q+l-j}(\frac{\ol z}{\val z})^{(l-j)} i^{j-l} \frac{1}{(q!)^2}
\frac{(r^2\vert  z\vert ^{2}/4  )^{q+\frac{(l-j)}{2}}}{(q+1)\cdots(q+l-j)}
 \widehat F^{1,3}((-j,z,0))dz.
\end{eqnarray*}

 \end{lemma} 
\begin{proof} 
For $F\in L^1(G_1)$, such that $\widehat F^3\in \CC{\T\times \C\times \R} $ and $l,j\in\Z, l+\la_k\geq 0, j+\la_k\geq 0, $  we have that
 \begin{eqnarray*}\label{pila al et de} 
 & &
 \langle\pi_{\la_k,\al_k}(F)V_k(\ch_j),V_k(\ch_l)\rangle\\
 \nn  &= &
\al_k\int_{\C}\int_\T\int_\C \int_\R e^{i\la_k \th}e^{i \al_k t}e^{-\frac{\al_k}{4}\vert  z\vert ^2}e^{-\al_k/2  w\ol z}\\
\nn  & &
F((e^{i\th},z,t))V_k(\ch_j)(e^{-i\th}w+e^{-i\th}z) dzdtd\th \ol{V_k(\ch_l)(w)} e^{-\pi\al_k \val w^2}dwd\th dzdt\\
\nn  &= &
\al_k i^{l-j}\sqrt {\frac{(\al_k/2)^{\la_k+j}}{(j+\la_k)!}}\sqrt 
{
\frac{(\al_k/2)^{\la_k+l}}
{(l+\la_k)!}
}\int_{\C}\int_\T\int_\C \int_\R e^{i\la_k \th}e^{i \al_k t}e^{-\frac{\al_k}{4}\vert  z\vert ^2}e^{-\al_k/2  w\ol z}\\
\nn  & &
F((e^{i\th},z,t))(e^{-i\th}w+e^{-i\th}z)^{j+\la_k} dzdtd\th \ol{w^{l+\la_k}} e^{-\pi\al_k \val w^2}dwd\th dzdt\\
\nn  &= &
\al_k i^{l-j}\sqrt {\frac{(\al_k/2)^{\la_k+j}}{(j+\la_k)!}}\sqrt 
{
\frac{(\al_k/2)^{\la_k+l}}
{(l+\la_k)!}
}\int_{\C}\int_\T\int_\C e^{-ij \th}e^{-\frac{\al_k}{4}\vert  z\vert ^2}e^{-\al_k/2  w\ol z}\\
\nn  & &
\widehat F^{3}((e^{i\th},z,\la_k))(w+z)^{j+\la_k} dzdtd\th \ol{w^{l+\la_k}} e^{-\pi\al_k \val w^2}dwd\th dz\\
\nn  & =&
\al_k i^{l-j}\sqrt {\frac{(\al_k/2)^{\la_k+j}}{(j+\la_k)!}}\sqrt 
{
\frac{(\al_k/2)^{\la_k+l}}
{(l+\la_k)!}
}\sum_{m=0}^\iy\sum_{q=0}^{\la_k+j}\binom{j+\la_k}q\int_{\C}\int_\T e^{-ij\th} \frac{(-\al_k /2 \ol z)^{m}}{m!}z^{j+\la_k-q}e^{-\frac{\al_k}{4}\vert  z\vert ^2}\\
\nn  & &
\widehat F^{3}((e^{i\th},z,\al_k))\int_\C   
\ol w^{l+\la_k}w^{m+q} e^{-\pi\al_k \val w^2}dwdzd\th\\
\nn  & &
\text{(orthogonality relations }\Rightarrow m=l+\la_k-q)\\
\nn  & =&
 i^{l-j}\sum_{\overset{l+\la_k\geq q}{0\leq q\leq j+\la_k}}\sqrt {\frac{(\al_k/2)^{\la_k+j}}{(j+\la_k)!}}\sqrt 
{
\frac{(\al_k/2)^{\la_k+l}}
{(l+\la_k)!}
}\binom{j+\la_k}q\int_{\C}\int_\T e^{-ij\th} \frac{(-\al_k /2 \ol z)^{l+\la_k-q}}{(l+\la_k-q)!}z^{j+\la_k-q}e^{-\frac{\al_k}{4}\vert  z\vert ^2}\\
\nn  & &
\widehat F^{3}((e^{i\th},z,\al_k))\frac{(\la_k+l)!}{(\al_k/2)^{l+\la_k}}dzd\th
\end{eqnarray*}
\begin{eqnarray*}
\nn  & =&
\sum_{\overset{l+\la_k-q\geq0}{0\leq q\leq j+\la_k}}
(-1)^{\la_k+l-q} i^{l-j}\frac{{\sqrt 
{
{(l+\la_k)!}
{(j+\la_k)!}
}}}{q!(j+\la_k-q)!}\frac{(\al_k/2 )^{\la_k-q+(j+l)/2}}{(l+\la_k-q)!}\\
\nn  & &
\int_\T e^{-ij\th}\int_{\C} (\ol z)^{l-j}\vert  z\vert ^{2(j+\la_k-q)}
e^{-\frac{\al_k}{4}\vert  z\vert ^2}\widehat F^{3}((e^{i\th},z,\al_k))dzd\th\\ 
\nn  & &
(q\to \la_k+j-q)\\
\nn  & =&
\sum_{\overset{l-j+q\geq 0}{0\leq q\leq j+\la_k}}
(-1)^{q+l-j} i^{l-j}\frac{{\sqrt 
{
{(l+\la_k)!}
{(j+\la_k)!}
}}}{q!(j+\la_k-q)!}\frac{(\la_k\al_k/2 )^{q+(l-j)/2}}{\la_k^{q+(l-j)/2}(q+l-j)!}\\
\nn  & &
 \int_{\C}(\ol z)^{l-j} \vert  z\vert ^{2q}e^{-\frac{\al_k}{4}\vert  z\vert ^2}\widehat F^{1,3}((-j,z,\al_k))dz\\ 
 \nn  &= &
\sum_{\overset{l-j+q\geq 0}{0\leq q\leq j+\la_k}}
(-1)^{q+l-j} i^{l-j}\frac{{\sqrt {(\la_k+j+1)\cdots (\la_k+l) }(j+\la_k-q+1 )\cdots (j+\la_k))
}}{\la_k^{q+(l-j)/2}(q!)^2}\frac{(\la_k\al_k/2 )^{q+(l-j)/2}}{(q+1)\cdots(q+l-j)}\\
\nn  & &
\int_{\C}(\frac{\ol z}{\val z})^{(l-j)} \vert  z\vert ^{2q+l-j}e^{-\frac{\al_k}{4}\vert  z\vert ^2}\widehat F^{1,3}((-j,z,\al_k))dz.
 \end{eqnarray*}

Furthermore, we have that:
\begin{eqnarray}\label{pir et de}
\nn &&
\sum_{\overset{l-j+q\geq 0}{0\leq q\leq +\iy}}  \int_{\C}
(-1)^{q+l-j}(\frac{\ol z}{\val z})^{(l-j)} i^{l-j} \frac{1}{(q!)^2}\frac{(r^2\vert  z\vert ^{2}/4  )^{q+\frac{(l-j)}{2}}}{(q+1)\cdots(q+l-j)}
 \widehat F^{1,3}((-j,z,0))dz\\
\nn   &= &\int_\T \int_{\C}i^{l-j} 
(\frac{-\ol z}{\val z})^{(l-j)} \sum_{\overset{l-j+q\geq 0}{0\leq q\leq +\iy}}  
(-1)^{q}\frac{1}{(q!)^2}\frac{(r^2\vert  z\vert ^{2}/4  )^{q+\frac{l-j}{2}}}{(q+1)\cdots(q+l-j)}
 e^{-ij\th}\widehat F^{3}((e^{i\th},z,0))dzd\th\\
 \nn  \nn  &= &i^{l-j} \int_\T e^{-ij\th} \int_{\C}(\frac{\ol z}{\val z})^{(l-j)}
J_{l-j}(-r\val z) 
 \widehat F^{3}((e^{i\th},z,0))dzd\th\\
   \nn  &= &i^{l-j} \int_\T e^{-ij\th} \int_{\C}(\frac{\ol z}{\val z})^{(l-j)}
\frac{i^{j-l}}{\pi}\int_0^{\pi}e^{i(-r\val z\cos(\ph))}\cos((l-j)\ph)d\ph
 \widehat F^{3}((e^{i\th},z,0))dzd\th\\
   \nn  &= &i^{l-j} \int_\T e^{-ij\th} \int_{\C}(\frac{\ol z}{\val z})^{(l-j)}
\frac{i^{j-l}}{2\pi}\int_{-\pi}^{\pi}e^{-i\Re(r\val ze^{i\ph})}e^{-i(l-j)\ph}d\ph
 \widehat F^{3}((e^{i\th},z,0))dzd\th\\
    \nn  &\underset{z=e^{i\nu(z)}\val z}{=}& \int_\T e^{-ij\th} \int_{\C}
\frac{1}{2\pi}\int_{-\pi}^{\pi}e^{-i\Re(r\ol z e^{i\nu(z)}e^{i\ph})}e^{-i(l-j)\nu(z)}e^{-i(l-j)\ph}d\ph
 \widehat F^{3}((e^{i\th},z,0))dzd\th\\
  \nn  &{=}& \int_\T e^{-ij\th} \int_{\C} 
\frac1{2\pi}\int_{-\pi}^{\pi}e^{-i\Re(r\ol z e^{i\ph})}e^{-i(l-j)\ph}d\ph
 \widehat F^{3}((e^{i\th},z,0))dzd\th\\
 \nn  &=& \frac{1}{2\pi}\int_{-\pi}^{\pi}\int_\T e^{-ij\th}
e^{-i(l-j)\ph} \widehat F^{2,3}((e^{i\th},re^{i\ph},0))d\ph d\th\\
 \nn  &=& \frac{1}{2\pi}\int_\T \int_\T e^{ij\th}
e^{-il\ph}\widehat F^{2,3}((e^{i(\ph-\th)},re^{i\ph},0))d\ph d\th\\
  &= & 
\langle \pi_{r}(F)(\ch_j),\ch_l\rangle .
 \end{eqnarray}
 \end{proof}
   
   \begin{remark}\label{limitsexplained}
\rm    Let $(\pi_{\la_k,\al_k})_k$ be a properly converging sequence in $\wh G_1 $. By Theorem \ref{limitset}  there exists  a convergent subsequence 
 (for  simplicity  of notations we  denote it also by $\la_k$ and $\al_k$)
 such that $(\alpha_k)_{k \in \N}$ tends to $0$ and that $\underset{k\to\infty}{\lim}\la_k\al_k=\om $, where $\om=\frac{r^2}{2}, r>0, $ or 
 $\underset{k\to\infty}{\lim}\la_k\al_k=0$.  If $\om>0 $, then, if $\al_k$ is positive
(respectively negative) for every  $k$,  we have that  $ \la_k>0$  (resp. $\la_k<0 $) for $ k$  large enough. 
 \end{remark}

\begin{lemma}\label{lim lakles j} 
Let $(\pi_{\la_k,\al_k})_k $ be a properly converging sequence in $\wh G_1 $. 
Suppose that  $ \underset{k\to\infty}{\lim} \la_k\al_k=\frac{r^2}{2}> 0$. 
Let $ F\in \l1{G_1}$, such that $ \widehat F^{1,3}\in C_c^\iy(\T\ti\C^2).$  Then we have that  
$$ \limk \noop{\pi_{\la_k,\al_k}(F)\circ V_k\circ(\Id- P_{\la_k,\sqrt{\vert \la_k\vert }})}=0.$$
Furthermore,  we have that 
\begin{eqnarray*}
\limk \noop{\pi_r(F)\circ(\Id- P_{\la_k,\sqrt{\vert \la_k\vert }})}=0
 \end{eqnarray*}

 \end{lemma}
\begin{proof} Suppose that $\al_k>0 $ for $k\in\N.  $ By (\ref{pilaal F}) we have for 
$ j\in \Z,\ \val j\geq \sqrt{\la_k},\  j\geq -\la_k$ that
\begin{eqnarray*}
 \pi_{\la_k,\al_k}(F)b_{\la_k+j}(w)&=&
\int_{\T}\int_{\C}e^{-ij \th}e^{-\frac{\al}{4}\vert  z\vert ^{2} +\frac{\al}{4}\vert w\vert ^2-
\frac{\al}{2} \Im( w\ol z)}\widehat F^3{(e^{i\th},z-w,\al)}b_{\la_k+j}(z) dzd\th\\
\nn  &= &
\int_{\C}e^{-\frac{\al}{4}\vert  z\vert ^{2} +\frac{\al}{4}\vert w\vert ^2-\frac{\al}{2} \Im( w\ol z)}\widehat 
F^{1,3}{(j,z-w,\al)}b_{\la_k+j}(z) dz\\
\nn  &= &
\pi_{\al_k}(\widehat F^{1,3}(-j))b_{\la_k+j}(w).
 \end{eqnarray*}
 Now, since $ \no{\widehat F^{1,3}(j)}_1\leq \frac{C_F}{(1+\vert j\vert )^4}, j\in\Z,$ for some $ C_F>0$  it follows that for $ k$  large enough, for any
for $$ \et=\sum_{\val j>\sqrt{\la_k}, j\geq -\la_k}c_{j}\ch_{\la_k+j}$$  we have  that 
\begin{eqnarray*}
 \no{\pi_{\la_k,\al_k}(F)\et}_2\nn  &= &
\no{\sum_{j}c_j \pi_{\al_k}(\widehat F^{1,3}(-j))b_{\la_k+\sqrt{\la_k}+j}}_2\\
\nn  &\leq &
(\sum_{j}\val {c_j}^2)^{1/2}(\sum_{j}\frac{C_F}{(\sqrt{\la_k}^2)(1+ \vert j\vert )^2})\\ 
\nn  &\leq &
\frac{1}{\sqrt{\la_k}}\no{\et}_2.  
 \end{eqnarray*}
Similarly for $ \pi_r(F)$:
\begin{eqnarray*}
 \pi_r(F)(\ch_j)(e^{i\th})\nn  &= &
\int_\T\widehat F^{2,3}(e^{i(\mu-\th)},e^{i\mu}r,0)e^{ij\mu}d\mu.
 \end{eqnarray*}
Then
\begin{eqnarray*}
 \val{\pi_r(F)(\ch_j)(e^{i\th})}\leq \frac{C_F}{(1+\val j)^{4}},\ j\in\Z,\ \th\in [ 0,2\pi[,
 \end{eqnarray*}
and so for $ k$  large enough and any $ \et=\sum_{\val j>\sqrt{\la_k}, j\geq -\la_k}c_{j}\ch_{\la_k+j}$ we have that
\begin{eqnarray*}
 \no{\pi_r(F)\et}_2\leq \frac{1}{\sqrt{\la_k}}\no\et_2.
 \end{eqnarray*}

\end{proof}

\begin{theorem}\label{estimates of sequences rne 0}  Let $(\pi_{\la_k,\al_k})_k $ be a properly converging sequence in $\wh G_1 $. 
Suppose that $ \underset{k\to\infty}{\lim} \la_k\al_k=\frac{r^2}{2}> 0$. 
Define :
\begin{eqnarray*}
 \sigma_{r,k}(a):=V_k\circ \pi_{r}(a)\circ V_k^*, a\in C^*(G_1), k\in\N.
 \end{eqnarray*}
Then  we have that 
\begin{eqnarray*}
 \limk \noop{\pi_{\la_k,\al_k}(a)-\sigma_{r,k}(a)}=0, a\in C^*(G_1).  
 \end{eqnarray*}
 \end{theorem}
\begin{proof} Suppose that $\al_k>0 $ for all $k\in\N $ (the case $\al_k<0 $ is similar). 
For   $ k$  large enough  it follows for any $l\geq j\in\Z $ and  $q\geq 0,  q+l-j\geq 0$ that: 
\begin{eqnarray*}
 \nn  & &
\frac{{\sqrt {(\la_k+j+1)\cdots (\la_k+l) }(j+\la_k-q+1 )\cdots (j+\la_k))
}}{\la_k^{q+(l-j)/2}(q!)^2}\frac{(\la_k\al_k/2 )^{q+(l-j)/2}}{(q+1)\cdots(q+l-j)}\val z^{2q}\\
\nn  \nn  &\leq &
\frac{\Big((1+\frac{l}{\la_k})(\frac{r^2}{4}+1)\Big)^{q+(l-j)/2}\val z^{2q}}{q!^2(q+1)^{l-j}}\\
\nn  \nn  &\leq &
\frac{\Big(\frac{r^2}{2}+2)^{q+(l-j)/2}\val z^{2q}}{q!^2(q+1)^{l-j}}.
 \end{eqnarray*}
 Let $ l,j\in \Z,\  l\geq -\la_k,\  \vert l\vert \leq \sqrt{\la_k},\  j\geq -\la_k,\  \vert j\vert \leq \sqrt{\la_k}$ . Then by (\ref{pila al et de}) and (\ref{pir et de}):
\begin{eqnarray*}
\nn  & &
\vert \langle\pi_{\la_k,\al_k}(F)V_k(\ch_j),V_k(\ch_l)\rangle-\langle \pi_{r}(F)(\ch_j),\ch_l\rangle \vert \\
 \nn  &=&\Big\vert 
\sum_{\overset{l-j+q\geq 0}{0\leq q\leq j+\la_k}} 
(-1)^{q+l-j} i^{l-j}\frac{{\sqrt {(\la_k+j+1)\cdots (\la_k+l) }(j+\la_k-q+1 )\cdots (j+\la_k))
}}{\la_k^{q+(l-j)/2}(q!)^2}\frac{(\la_k\al_k/2 )^{q+(l-j)/2}}{(q+1)\cdots(q+l-j)}\\
\nn  & &
\int_\T e^{-ij\th}\int_{\C}(\ol z)^{l-j} \vert  z\vert ^{2q}e^{-\frac{\al_k}{4}\vert  z\vert ^2}\widehat F^{3}((e^{i\th},z,\al_k))dzd\th\\
\nn  &- &
\sum_{\overset{l-j+q\geq 0}{0\leq q\leq +\iy}}   \int_{\C}
(-1)^{q+l-j}e^{-ij\th}(r \ol z/2 )^{(l-j)} i^{l-j}\frac{1}{(q!)^2}\frac{(r^2\vert  z\vert ^{2}/4  )^{q}}{(q+1)\cdots(q+l-j)}
 \widehat F^{1,3}((j,z,0))dz\Big\vert 
 \end{eqnarray*}
 Then, 
 \begin{eqnarray*}
  \nn  & &\vert \langle\pi_{\la_k,\al_k}(F)V_k(\ch_j),V_k(\ch_l)\rangle-\langle \pi_{r}(F)(\ch_j),\ch_l\rangle \vert \\
  \nn  &= &\Big\vert 
\sum_{\overset{l-j+q\geq 0}{0\leq q\leq j+\la_k}} 
(-1)^{q}\frac{{\sqrt {(1+\frac{j+1}{\la_k})\cdots (1+\frac{l}{\la_k}) }(1+\frac{j-q+1 }{\la_k})\cdots (1+\frac{j}{\la_k}))
}}{(q!)^2}\frac{(\la_k\al_k/2 )^{q+(l-j)/2}}{(q+1)\cdots(q+l-j)}\\
\nn  & &
 \int_{\C}(\ol z)^{l-j} \vert  z\vert ^{2q}e^{-\frac{\al_k}{4}\vert  z\vert ^2}\widehat F^{1,3}((j,z,\al_k))dz\Big\vert \\
\nn  &- &
\sum_{\overset{l-j+q\geq 0}{0\leq q\leq +\iy}}  \int_{\C} 
(-1)^{q}(r \ol z/2 )^{l-j}\frac{1}{(q!)^2}\frac{(r^2\vert  z\vert ^{2}/4  )^{q}}{(q+1)\cdots(q+l-j)}
 \widehat F^{1,3}((j,z,0))dz\Big\vert \\
  \nn  &\leq&
\sum_{\overset{l-j+q\geq 0}{0\leq q\leq \la_k+j}} 
\Big\vert \frac{{\sqrt {(1+\frac{j+1}{\la_k})\cdots (1+\frac{l}{\la_k}) }(1+\frac{j-q+1 }{\la_k})\cdots (1+\frac{j}{\la_k}))
}}{(q!)^2}\frac{(\la_k\al_k/2 )^{q+(l-j)/2}}{(q+1)\cdots(q+l-j)}\\
\nn  &- &
\frac{1}{(q!)^2}\frac{(r ^{2}/4  )^{q+(l-j)/2}}{(q+1)\cdots(q+l-j)}\Big\vert 
 \int_{\C}\val z^{l-j}\val z^{2q}\vert \widehat F^{1,3}((j,z,\al_k))\vert dz\\
 \nn  &+ &
\sum_{\overset{l-j+q\geq 0}{q= \la_k+j+1}\geq0}^\iy  \int_{\C}
\frac{1}{(q!)^2}\frac{(r^2\vert  z\vert ^{2}/4  )^{q+(l-j)/2}}{(q+1)\cdots(q+l-j)}
 \vert \widehat F^{1,3}((j,z,\al_k))\vert dz
\\
&+&\sum_{\overset{l-j+q\geq 0}{0\leq q\leq \la_k+j}} 
\Big\vert \frac{{\sqrt {(1+\frac{j+1}{\la_k})\cdots (1+\frac{l}{\la_k}) }(1+\frac{j-q+1 }{\la_k})\cdots (1+\frac{j}{\la_k}))
}}{(q!)^2}\frac{(\la_k\al_k/2 )^{q+(l-j)/2}}{(q+1)\cdots(q+l-j)}
\\
& &\int_\C(\ol z)^{l-j} \vert  z\vert ^{2q}\left(e^{-\frac{\al_k}{4}\vert  z\vert ^2}-1\right)\widehat F^{1,3}((j,z,\al_k))dz\Big\vert\\
\nn  & +&
 \sum_{\overset{l-j+q\geq 0}{q= 0}}^\iy \int_{\C}\frac{1}{(q!)^2}\frac{(r^2\vert  z\vert ^{2}/4  )^{q+(l-j)/2}}{(q+1)
 \cdots(q+l-j)}\Big \vert\widehat F^{1,3}((j,z,\al_k))-\widehat F^{1,3}((j,z,0))\Big \vert dz \\
 \nn  &\leq &
 \Big\vert (\la_k\al_k/2-{r^2/4})\Big\vert\Big(
\sum_{\overset{l-j+q\geq 0}{1\leq q\leq \la_k+j}} (q+(l-j)/2)
\frac{{\sqrt {(1+\frac{j+1}{\la_k})\cdots (1+\frac{l}{\la_k}) }(1+\frac{j-q+1 }{\la_k})\cdots (1+\frac{j}{\la_k}))
}}{(q!)^2(q+1)\cdots(q+l-j)} (r^2/4+1)^{q-1+(l-j)/2}\\
\nn  & &
 \int_{\C}\val z^{l-j}\val z^{2q}\vert \widehat F^{1,3}((j,z,\al_k))\vert dz\Big)\\
 \nn  &+ &\
\sum_{\overset{l-j+q\geq 0}{ q\geq \la_k+j+1}}  \int_{\C}
\frac{1}{(q!)^2}\frac{(r^2\vert  z\vert ^{2}/4  )^{q+(l-j)/2}}{(q+1)\cdots(q+l-j)}
 \vert \widehat F^{1,3}((j,z,\al_k))\vert dz
\Big \vert \\  
\nn  &+ &
\Big\vert \sum_{0\leq q\leq j+\la_k}
\frac{{\sqrt {(\la_k+j+1)\cdots (\la_k+l) }(j+\la_k-q+1 )\cdots (j+\la_k))
}}{\la_k^{q+(l-j)/2}(q!)^2}\frac{(\la_k\al_k/2 )^{q+(l-j)/2}}{(q+1)\cdots(q+l-j)}\\
\nn  & &
\int_{\C}(\ol z)^{l-j} \vert  z\vert ^{2q}\Big(e^{-\frac{\al_k}{4}\vert  z\vert ^2}-1\Big)\widehat F^{1,3}((j,z,\al_k))dz\Big \vert \\
\nn  & +&
 \sum_{\overset{l-j+q\geq 0}{q= 0}}^\iy \int_{\C}\frac{1}{(q!)^2}\frac{(r^2\vert  z\vert ^{2}/4  )^{q+(l-j)/2}}{(q+1)
 \cdots(q+l-j)}\Big \vert\widehat F^{1,3}((j,z,\al_k))-\widehat F^{1,3}((j,z,0))\Big \vert dz
 \end{eqnarray*}
 \begin{eqnarray*} 
 \nn  &\leq &
 \Big\vert (\la_k\al_k/2-{r^2/4})\big\vert \int_{\C}
\sum_{0\leq q<\iy}\frac{2^q(r^2/4+1)^{q}\val z^{2q}}{q!}
 \widehat F^{1,3}((j,z,\al_k))\vert dz\Big)\\
 \nn  &+ &\
\frac{1}{j+\la_k+1}\int_{\C}\sum_{q=0}^\iy  
\frac{2^q(r^2/4+1)\vert  z\vert ^{2q}  }{q!}
 \vert \widehat F^{1,3}((j,z,\al_k))\vert dz
 \\   
\nn  &+ &
2\al_k e\int_{\C}\sum_{q=0}^\iy  
\frac{2^q(r^2/4+1)\vert  z\vert ^{2q}  }{q!}
 \vert \widehat F^{1,3}((j,z,\al_k))\vert dz
+\int_{\C}\sum_{q=0}^\iy  
\frac{2^q(r^2/4+1)\vert  z\vert ^{2q}  }{q!}
 \vert \widehat F^{1,3}((j,z,\al_k))-\widehat F^{1,3}((j,z,0))\vert dz.
\end{eqnarray*} 
Hence
\begin{eqnarray*}
 \nn  & &
\vert \langle\pi_{\la_k,\al_k}(F)V_k(\ch_j),V_k(\ch_l)\rangle-\langle \pi_{r}(F)(\ch_j),\ch_l\rangle \vert \\
\nn  &\leq &
(\vert (\la_k\al_k/2-{r^2/4})\vert +\frac{2}{\la_k}+2e\al_k)\int_\C e^{(r^2/2+2)\val z^2}\vert\widehat F^{1,3}(j,z,\al_k)\vert dz\\
&+&\int_\C e^{(r^2/2+2)\val z^2}\vert\widehat F^{1,3}(j,z,\al_k)-\widehat F^{1,3}(j,z,0)\vert dz\\
\nn  &\leq &
\de_k\frac{1}{(1+\val j)^4}
\end{eqnarray*}

for some sequence $(\de_k)_k $ in $\R_+ $ with $\limk\de_k=0 $.
Replacing $ F$  by $ F^*$  we see that also
\begin{eqnarray*}
 \vert \langle\pi_{\la_k,\al_k}(F)V_k(\ch_j),V_k(\ch_l)\rangle-\langle \pi_{r}(F)(\ch_j),\ch_l\rangle \vert\leq \frac{\de_k}{(1+\val l)^4},k\in\N. \\
 \end{eqnarray*}
Finally
\begin{eqnarray*}
 \vert \langle\pi_{\la_k,\al_k}(F)V_k(\ch_j),V_k(\ch_l)\rangle-\langle \pi_{r}(F)(\ch_j),\ch_l\rangle \vert\leq \frac{\de_k}{(1+\val l)^2(1+\val j)^2},k\in\N, \val l,\val j\leq \sqrt{\la_k}.
 \end{eqnarray*}

Hence for any $\et_k=\sum_{j=-\sqrt{\la_k}}^{\sqrt{\la_k}}c_j^k\ch_j,\ps_k=\sum_{j=-\sqrt{\la_k}}^{\sqrt{\la_k}}d_j^k\ch_j\in\l2\T $ we have that 
\begin{eqnarray*} 
 \nn  & &
\val{
\langle \pi_{\la_k,\al_k}(F)V_k(\et_k),V_k(\ps_k)\rangle-\langle \pi_r(F)\et_k,\ps_k\rangle
}\\
 &\leq &
C\no {\et_k}_2 \no {\ps_k}_2 \de_k
\end{eqnarray*}
for $ C:=\sqrt{  
\sum_{j=0}^\iy\frac{1}{(1+j)^4}
}$ .
This shows together with Lemma \ref{lim lakles j} that
\begin{eqnarray*}
 \limk \noop{\pi_{\la_k,\al_k}(F)\circ V_k- V_k\circ\sigma_{r,k}(F)}=0.
 \end{eqnarray*}
Hence 
 \begin{eqnarray*}
 \limk \noop{\pi_{\la_k,\al_k}(F)- \sigma_{r,k}(F)}=0.
 \end{eqnarray*}

\end{proof} 
\subsection{Convergence to characters.}

If the sequence $ (\la_k\al_k)_k$  tends to 0,  $\al_k>0 $ (resp. $\al_k<0 $) for $k\in\N $ and the sequence $ (\val{\la_k})_k$ has  a bounded subsequence, 
then we find a  subsequence such that $ \la_k=\la_\iy(\in\Z)$  for all $ k$. The limit set $ L$  of this subsequence is according to Theorem \ref{limitset} given by
\begin{eqnarray*}
 L=\{\ch_j;\ j\in\Z,\ j\leq\la_\iy\}\ (\text{ resp. }L=\{j\in\Z, j\geq \la_\iy\}).
 \end{eqnarray*}
 If the sequence $(\vert \la_k\vert )_k $ is unbounded,  the fact that the limit set of the sequence $(\pi_{\la_k,\la_k})_k $ 
 is not empty forces $\limk \la_k=+\iy=:\la_\iy$ if $\al_k>0 $ (resp. $\limk \la_k=-\iy=:\la_\iy $ if $\al_k<0, k\in\N $). 
 
Let for $\la_\iy\in\Z \cup \{+\iy,-\iy\} $ 
\begin{eqnarray*}
 \pi_{\la_\iy,0} :=\oplus_{\la\leq \la_\iy}\ch_\la \ (\text{ resp. } \pi_{\la_\iy,0} :=\oplus_{\la\geq \la_\iy}\ch_\la ))
 \end{eqnarray*}
be the direct sum of the characters $\ch_\la, \la\leq \la_\iy  $ (resp. of the characters $\la\geq \la_\iy $).

\begin{theorem}\label{estimates of sequences r 0}  
 Let $(\pi_{\la_k,\al_k})_k $ be a properly converging sequence in $\wh G_1 $. 
Suppose that  $ \underset{k\to\infty}{\lim} \al_k=0$ and that $\limk \al_k \la_k=0 $.
Define:
\begin{eqnarray*}
 \sigma_{\la_\iy,k}(a):=V_k\circ \pi_{\la_\iy,0}(a)\circ V_k^*, a\in C^*(G_1), k\in\N.
 \end{eqnarray*}

Then for every $ a\in C^*(G_1)$ we have that
\begin{eqnarray*}
 \limk \noop{\pi_{\la_k,\al_k}(a)-\sigma_{\la_\iy,k}(a)}=0.  
 \end{eqnarray*}
 \end{theorem}
\begin{proof} We consider only the case $\al_k>0, k\in \N $. Suppose first that $ \widehat F^{1,3}$ (and so also $ (\widehat F^*)^{1,3}$) has finite support in the first variable $ j$ and compact support in the variable $ z$. 
  Then we have for some $ m\in\N$ and some compact ball  $K= B_R\subset \C$ that 
  \begin{eqnarray*}
 \vert \widehat F^{1,3}(j,z,\la_k)\vert \leq {C_F} 1_{[ -m,m]}(j)1_K(z), j\in\Z,z\in\C,\al\in\R.
 \end{eqnarray*} 
 
 Let  $ l,j\in \Z, l >j , l\geq -\la_\iy,   j\geq -\la_\iy$. 
 Lemma \ref{pik et de} implies that
 \begin{eqnarray}\label{r=o id} 
 \nn& &
 \langle\pi_{\la_k,\al_k}(F)V_k(\ch_j),V_k(\ch_l)\rangle\\
\nn  & =&
\sum_{\overset{l-j+q\geq 0}{0\leq q\leq j+\la_k}}
(-1)^{q+l-j} i^{l-j}\frac{{\sqrt {(\la_k+j+1)\cdots (\la_k+l) }(j+\la_k-q+1 )\cdots (j+\la_k))
}}{\la_k^{q+(l-j)/2}(q!)^2}\frac{(\la_k\al_k/2 )^{q+(l-j)/2}}{(q+1)\cdots(q+l-j)}\\
\nn  & &
\int_{\C}(\frac{\ol z}{\val z})^{(l-j)} \vert  z\vert ^{2q+l-j}e^{-\frac{\al_k}{4}\vert  z\vert ^2}\widehat F^{1,3}((-j,z,\al_k))dz.
 \end{eqnarray}
 In particular we see that
 \begin{eqnarray*}
 \langle\pi_{\la_k,\al_k}(F)V_k(\ch_j),V_k(\ch_l)\rangle=0,
 \end{eqnarray*}
if $\val j >m$ or $\val j>m $.
It follows then that 
\begin{eqnarray*}
 \val{\langle\pi_{\la_k,\al_k}(F)V_k(\ch_j),V_k(\ch_l)\rangle}\nn  &\leq &
\sum_{\overset{l-j+q\geq 0}{0\leq q\leq j+\la_k}}\vert \la_k\al_k\vert^{q+(l-j)/2}(1+\frac{m}{\vert \la_k\vert })^{q+(l-j)/2}\frac{1}{(q!)^2} 
\\  
\nn  & & 
 {C_F} R ^{2q+l-j}1_{[ -m,m]}(j)1_{[ -m,m]}(l)\int_{\C}e^{-\frac{\al_k}{4}\val z^2}  1_K(z)dz\\
\nn  &\leq &
\vert \la_k\al_k\vert^{(l-j)/2}R^{l-j}C_F\int_{K}dz 1_{[ -m,m]}(j)1_{[ -m,m]}(l)
\sum_{q=0}^{\iy}(\vert \la_k\al_k\vert(1+\frac{m}{\vert \la_k\vert })R)^{q}\frac{1}{(q!)} 
 \end{eqnarray*} 
(similarly if $j>l $.). For $l=j\in \Z,  l\geq -\la_\iy$ we see that 
  \begin{eqnarray}\label{r=o idequal} 
 \nn& &
 \langle\pi_{\la_k,\al_k}(F)V_k(\ch_j),V_k(\ch_j)\rangle\\
\nn  & =&\int_{\C} e^{-\frac{\al_k}{4}\vert  z\vert ^2}\widehat F^{1,3}((-j,z,\al_k))dz+\sum_{\overset{q> 0}{ q\leq j+\la_k}}
(-1)^{q} \frac{{ 
{
{(j+\la_k)!}
}}}{q!(j+\la_k-q)!}\frac{(\la_k\al_k/2 )^{q}}{\la_k^{q}q!}\\
\nn  & &
 \int_{\C} \vert  z\vert ^{2q}e^{-\frac{\al_k}{4}\vert  z\vert ^2}\widehat F^{1,3}((-j,z,\al_k))dz
 \end{eqnarray}
 and that
 \begin{eqnarray}\label{pir et de}
\langle \pi_{\la_\iy,0}(F)(\ch_j),\ch_l\rangle
\nn  &= &
\int_\T\int_\T \widehat F^{2,3}(\th,0,0)e^{i j(\mu-\th)}e^{-il \mu}d\th d\mu\\
\nn  &= &
\int_\T \widehat F^{1,2,3}(-j,0,0)e^{i (j-l)\mu} d\mu\\
\nn  &= &
\de_{j,l}\widehat F^{1,2,3}(-j,0,0). 
 \end{eqnarray}

Hence for some constant $ D_F>0$ , we have that 
\begin{eqnarray*}
 \val{\langle\pi_{\la_k,\al_k}(F)V_k(\ch_j),V_k(\ch_l)\rangle-\de_{j,l}\widehat F^{1,2,3}(j,0,0)}&\leq& 
\val{\al_k\la_k} D_F1_{[ -m,m]}(j)1_{[ -m,m]}(l).
 \end{eqnarray*}

This shows  that
for any $ \et_k=\sum_{j\in\Z}c^k_j \ch_j$  we have that
\begin{eqnarray*}
 \nn  & &
\no{\pi_{\la_k,\al_k}(F)\circ V_k(\et_k)-V_k\circ \pi_{\la_\iy,0}(F)(\et_k)}_2^2\\
\nn  &= &
\sum_{\underset{j\geq -\la_\iy}{l\geq -\la_\iy}}c_j\ol{c_l}\langle (\pi_{\la_\iy,\al_k}(F)\circ V_k-\de_{j,l}\widehat F^{1,2,3}(-j,0,0))(V_k(\ch_j)),(\pi_{\la_\iy,\al_k}(F)\circ V_k-\de_{j,l}\widehat F^{1,2,3}(-l,0,0))(V_k(\ch_l))\rangle\\
\nn  &= &
\sum_{\underset{j\geq -\la_\iy}{l\geq -\la_\iy}}c_j\ol{c_l}\Big(\langle (\pi_{\la_k,\al_k}(F^*\ast F)\circ V_k(\ch_j),V_k(\ch_l)\rangle-
\de_{j,l}\ol{\widehat F^{1,2,3}}(-j,0,0)) 
\langle\pi_{\la_k,\al_k}(F)(V_k(\ch_j)),V_k(\ch_l)\rangle\\
\nn  &- &
\de_{j,l}\widehat F^{1,2,3}(l,0,0)\langle V_k(\ch_j),
\pi_{\la_k,\al_k}( F)( V_k(\ch_l)\rangle  + 
\de_{j,l}\widehat{(F^*\ast F)}^{1,2,3}(-j,0,0)\Big)\\
\nn  &\leq &
E_k\vert \al_k\la_k\vert \no{\et_k}^2
 \end{eqnarray*}
for some new constant $E_k>0  $. Therefore
\begin{eqnarray*}
 \limk \noop{\pi_{\la_\iy,\al_k}(F)-\sigma_{\la_\iy,k}(F)}=0.
 \end{eqnarray*}
Since these $ F$ 's are dense in $ C^*(G_1)$ the theorem follows.

\end{proof} 

  \subsection{The $C^*$-algebra of the group $G_1.$}
  \begin{definition}
  Let us recall that 
  \begin{eqnarray*}
 \widehat{G_1}&=&
 \{\pi_{\la,\al}\vert \la\in\Z,\al\in\R^*\}\bigcup \{\pi_r\vert r>0\}\bigcup \{\pi_\la\vert \la\in\Z\}\\
 \nn  &= &
\GA_2\cup \GA_1\cup\GA_0.
 \end{eqnarray*}

   Define for $c\in C^*(G_1)$ the Fourier transform $\F(c)$ of $c$ by 
   \begin{itemize}
   \item $\F(c)(\la,\al)=\pi_{\la,\al}(c)\in\B(\H_{\pi_{\la,\al}}),\ \ (\la,\al)\in \Ga_2.$
    \item $\F(c)(r)=\pi_{r}(c)\in\B(\H_{\pi_{r}}),\ \ r\in\Ga_1.$
    \item $\F(c)(\la)=\chi_\la(c)\in\C,\ \la\in\Ga_0.$
   \end{itemize}
  \end{definition}
\begin{definition}
 Let $\D_1$ be the family consisting of all operator fields $A\in\ell^\iy(\widehat{G_1})$ satisfying the following conditions 
 \begin{enumerate}
  \item $A(\ga)$ is a compact operator on $\H_{\pi_{\ga}}$ for every $\ga\in\g_1^\ddag/G_1.$
  \item The mapping $\g_1^\ddag/G_1\longrightarrow\B(\H_{\pi_\ga}):\ \ga\longmapsto A(\ga)$ 
  is norm continuous on $\GA_2,\GA_1 $ and on $\GA_0 $.
  \item $\underset{\ga\to\iy}{\lim}\Vert A(\ga)\Vert_{op}=0.$
  \item $\underset{r\to0}{\lim}\Vert A(r)-A(0)\Vert_{op}=0$ uniformly in $\la$,
  where
  \begin{eqnarray*}
 A(0):=\oplus_{\la\in\Z}A(\la).
 \end{eqnarray*}

  \item For every properly converging sequence $(\pi_{\la_k,\al_k})_k$ such that 
  $ \underset{k\to\infty}{\lim} \la_k\al_k=\frac{r^2}{2}> 0$, we have 
  $$\underset{k\to\iy}{\lim}\Vert A(\la_k,\al_k)-V_k\circ A(r)\circ V_k^*\Vert_{op}=0.$$
  For every properly converging sequence $(\pi_{\la_k,\al_k})_k$ such 
  that $\underset{k\to\iy}{\lim}\la_k=\la_\iy$ and  $\limk \al_k=0, \underset{k\to\infty}{\lim} \la_k\al_k=0$ and $\ve \al_k> 0, \ k\in\N, \ve =+1 \text{ for all k or } \ve=-1 \text{ for all k},  $
  and that 
   we have 
  $$\underset{k\to\iy}{\lim}\Vert A(\la_k,\al_k)-V_k\circ A(\la_\iy,0)\circ V_k^*\Vert_{op}=0,$$
 \end{enumerate}
where 
\begin{eqnarray*}
 A(\la_\iy,0):=\oplus_{ \la\leq \la_\iy} A(\la)\ (\text{if }\al_k>0 \text{ for all }k)\text{ resp. }A(\la_\iy,0):=\oplus_{ \la\geq \la_\iy} A(\la)\ (\text{if }\al_k<0 \text{ for all }k).
 \end{eqnarray*}

\end{definition}
As  a consequence of relation \ref{identity Vk}, Theorem \ref{estimates of sequences rne 0}, Theorem  \ref{estimates of sequences r 0} and Theorem 4.10 in \cite{Lud-Ell-Abd}, we can apply Theorem 3.5 in \cite{Lud-Reg1} and therefore we have 
\begin{theorem}
 The $C^*$-algebra of $G_1$ is isomorphic to $\D_1$ under the Fourier
transform and $C^*(G_1)$ fulfills  the \textit{NCDL} condition.
\end{theorem}
\section{The \textit{NCDL}-property of  $ C^*(G_n)$.  }\label{cst gn}
For $n\geq1$, let $G^n=G_1\times\cdots\times G_1,$ $n$ times.
\begin{theorem}
 The $C^*$-algebra of the group $G^n$ is isomorphic to the tensor product of $C^*(G_1)\ol\otimes\cdots\ol\otimes C^*(G_1)$ ($n$ times) and 
 $C^*(G^n)$ has the \textit{NCDL} condition.
\end{theorem}
\begin{proof} Since $C^*(G_1) $ is liminary, 
we have  by  Theorem \ref{thbijwhAtiwhBandwhAotiB},  that  $\wh{C^*(G^n)}\simeq\wh{C^*(G_1)}\times \cdots\times \wh{C^*(G_1)}$.
    
It suffices to apply Theorem \ref{tensor ncdl}.
\end{proof}

For 
$a\in C^*(G^n)$ the Fourier transform $\F $ is thus defined by 
 $$\F(a)(\pi_1\times \cdots\times\pi_n)=\wh a(\pi_1\times\cdots\times\pi_n)=\pi_1\otimes\cdots\otimes\pi_n(a)=,
\ (\pi_i)_{1\leq i\leq n}\subset\wh{C^*(G_1)}.$$ 
 In particular for elementary tensors  $a_1\otimes a_2\otimes\cdots\otimes a_n $ 
 we obtain then
 \begin{eqnarray*}
 \F(a)(\pi_1\times \cdots\times\pi_n)
 =\pi_1(a_1)\otimes\cdots \otimes \pi_n(a_n)\in\B(\H_{\pi_1}\otimes\cdots\otimes \H_{\pi_n}).
 \end{eqnarray*}

 We consider now the centre $ \ZZ^n$ of the groups $ G^n$.
 This subgroup $ \ZZ^n$  of $ G^n$  is given by 
 \begin{eqnarray*}
 \ZZ^n=\{(1,0,t_1)\times \cdots\times(1,0,t_n);\ t_1, \cdots, t_n\in\R \}.
 \end{eqnarray*}
Let  $ \{Z_1, ,\cdots, Z_n\}$  be the canonical basis of the centre $\z $ of the Lie algebra $\h_n $ of the group $\HH_n $. 
This means that
\begin{eqnarray*}
 Z_j=(0,0)\times\cdots (0,0)\times (0,1)\times (0,0)\cdots\times (0,0)\in \C^n\ti \R^n,
 \end{eqnarray*}
$(0,1) $ appearing at the $j $th position. We denote by $ \z_0$  the subspace
\begin{eqnarray*}
 \z_0:=\left\{\sum_{j=0} ^{n} z_j Z_j;\ \sum_{j=0}^{n }z_j=0\right\}
 \end{eqnarray*}
 of $\z $. Then $ \z_0$  is   of codimension 1 in $ \z$ and
 \begin{eqnarray*}
  \ZZ_0^n:=\exp(\z_0)=\left\{(1,0,t_1)\times\cdots \times(1,0,t_n); \sum_{j=1}^{n}t_j=0\right\}
 \end{eqnarray*}
 is  a closed connected and central subgroup of $ G ^{n}$.
 The quotient group $ G^{n}/\ZZ_0^n$ is then isomorphic to $ G_n=\T^{n}\ltimes \HH_n$. 
 To see this, it suffices to consider the canonical basis $ \{T_j,X_j,Y_j, Z_j; j=1,\cdots, n\}$ of $ \g^{n}$ 
 and compute the non trivial brackets in this Lie algebra modulo $ \z_0$. Let $ Z:=\frac{1}{n}(\sum_{j=1}^{n}Z_j)$.
 We have 
 \begin{eqnarray*}
  [ T_j,X_j]=Y_j,\  [ T_j,Y_j]=-X_j \text{ modulo }  \z_0,\\
\end{eqnarray*}
 \begin{eqnarray*}
 [ X_j,Y_j]=Z_j=Z+(Z_j-Z)=Z\text{ modulo }\z_0
 \end{eqnarray*}
since $ Z_j-Z\in\z_0$.

The spectrum of the group $ G_n$  can now be identified with the spectrum of  $ \widehat{G^{n}/\ZZ_0^n}$, which is the subset of $ \widehat{G^{n}}$  consisting of the $ n$ -tuples
$ \pi_1\times\cdots\times \pi_n\in \widehat{G^{n}}$ such that 
$ \pi_1\otimes\cdots\otimes \pi_n((1,0,z_1)\times\cdots\times(1,0, z_n))=\Id$,  if $ \sum_{j=1}^{n}z_j=0$. This means that 
\begin{eqnarray*}
 \widehat{G_n}&\simeq &
 \wh {G^n_0}=:\{((\la_1,\cdots, \la_n),(\al,\cdots, \al));\al\in\R^{*}, \la_j\in\Z, j=1,\cdots, n\}\bigcup \{\pi\in\widehat{G^{n}}, 
 \pi=\Id \text{ on }\ZZ_0^n\}.
 \end{eqnarray*}
Let \begin{eqnarray*}
 \K=\{a\in C^*(G^{n});\pi(a)=0,\ \forall \pi\in \widehat{G^{n}}\text{ which are trivial on }\ZZ_0^n\}.
 \end{eqnarray*} 
 \begin{theorem}
  The $C^*$-algebra of the Heisenberg motion group $C^*(G_n)$ is isomorphic to $C^*(G^n)/\K$ and  $C^*(G_n)$ satisfies the \textit{NCDL} condition. 
 \end{theorem}
\begin{proof}
The ideal  $ \K=\{a\in C^*(G^{n});\pi(a)=0,\ \forall \pi\in \widehat{G^{n}}\text{ which are trivial on }\ZZ_0^n\}$ is the kernel in $C^*(G^n)$ 
 of the canonical surjective  homomorphism $\delta^n: C^*(G^n)\longrightarrow C^*(G_n)$, which is defined on $\L1{G^n} $ by
 \begin{eqnarray*}
 \de^n(F)(g):=\int_{\ZZ_0^n}F(g z)dz, g\in G^n.
 \end{eqnarray*}
Then the $C^*$-algebra 
 of the Heisenberg motion group $C^*(G_n)$ is isomorphic to $C^*(G^n)/\K.$
 
 Let $\rh^n $ be the restriction map 
 \begin{eqnarray*}
 \left.\begin{array}{cccc}
 & & & \\
\rh^n: & l^\iy(\wh {G^n})& \longrightarrow& l^{\iy}(\wh{G^n_0})\\
&\ph&\mapsto&\rh^n(\ph)=
 \ph\res{\wh{G^n_0}}\end{array}\right.
 \end{eqnarray*}
for any uniformly bounded operator field $\ph $ defined on $\wh{G^n} $.

Then the $C^* $-algebra of the group $G_n $ can be identified with the sub-algebra   $\rh^n (\F(C^*(G^n)) )$ of the algebra  $l^\iy(\wh{G^n_0}) $.
In particular,  if we have a properly convergent sequence
$\ol{\ga}=(\ga_k)_k $ in $\wh {G_n} $, then we can write
\begin{eqnarray*}
 \ga_k=\pi^1_k\otimes\cdots\otimes \pi^n_k,k\in\N,
 \end{eqnarray*}
where $\ol{\ga^j}=(\pi^j_k)_k $ is  a properly convergent sequence in $\wh{G_1} $ with limit set $L_j $, $j=1,\cdots, n $. Let $\si_{\ol{\ga^j},k}, k\in\N, $ be the corresponding norm control. Then, 
\begin{eqnarray*}
 \si_{\ol \ga}=\si_{\ol{\ga^1}}\otimes\cdots \otimes \si_{\ol{\ga^n}}.
 \end{eqnarray*}

\end{proof}


\begin{thebibliography}{99}
\bibitem {Bal} B. Blackadar \textit{Operator Algebras theory of $C^*$-algebras and von Neumann algebras}. Springer-Verlag Berlin, Heidelberg, $2006$.
 

\bibitem{Be-Be-Lu} Ingrid Beltita, 
Daniel Beltita, 
Jean Ludwig,  \textit{Fourier Transforms of $ C^*$-Algebras of Nilpotent Lie Groups},
Int Math Res Notices (2017) (3), 677-714.


\bibitem   {Dixmier} J. Dixmier, \textit{Les $C^*$-alg\`ebres et leurs
repr\'esentations}, Gauthier-Villars.
\bibitem{E.L.}  M. Elloumi  and J. Ludwig, {\it Dual topology of the motion groups
$SO(n)\ltimes\mathbb{R}^n$},  Forum Math. 22 (2010), no. 2, 397-410.

\bibitem{th-Ell} Elloumi Mounir, \textit{Espaces duaux de certains produits semi-directs et noyaux associ\'es aux orbites plates}, Th\`ese universit\'e de Metz, 2009. 

\bibitem {Fell-1} Fell, J. M. G. \textit{The dual spaces of $C^*$-algebras}. Trans. Amer. Math. Soc. 94 1960 365-403.

 
\bibitem {Fo} G.B.Folland, \textit{Harmonic Analysis in Phase Space}, Princeton University Press, $1989$.

\bibitem{Hal-Rah} Ben Halima M. , Rahali A., \textit{Dual Topology Of the Heisenberg Motion Groups}, Indian P. Pure Appli. Math. 45 (4), 
(2014), 513-530.



   

\bibitem   {Lin-Lud} Y-F,Lin, J.Ludwig, \textit{The $C^*$-algebras of
ax+b-like  groups}, Journal of Functional Analysis, 259 (2010), 104-130.
\bibitem   {Lud-Ell-Abd} F.Abdelmoula, M.Elloumi, J.Ludwig, \textit{The $C^*$-algebra of the motion group $SO(n)\ltimes\R^n$}, Bull. Sci. math. 
$135 (2011) 166-177.$
\bibitem   {Lud-Reg1} J.Ludwig, H.Regeiba, \textit{$C^*$-Algebras with Norm Controlled Dual Limits and Nilpotent Lie Groups},   	
Journal of Lie Theory $25$ $(2015)$, No. $3,$ $613-655$.
\bibitem   {Lud-Reg2}  H.Regeiba, J.Ludwig, \textit{The $C^*$-algebra of some $6$-dimensional nilpotent Lie groups},
Advances in Pure and Applied Mathematics Volume $5$, issue $3$ (Aug $2014$).
\bibitem   {Lud-Tur} {J. Ludwig}, L. Turowska: {\it The $C^*$-algebras of the Heisenberg Group
and of thread-like Lie groups.} \textit{Math. Z.} 268 (2011), no. 3-4, 897-930.
\bibitem {Gun-Lud} G\" unther, Janne-Kathrin, Ludwig, Jean 
\textit{The $C^*$-algebras of connected real two-step nilpotent Lie groups}. 
Rev. Mat. Complut. $29\  (2016)$, no. $1,\ 13-57$. 
\bibitem   {Reg} H.Regeiba, \textit{Les $C^*$-alg\`ebres des groupes de Lie nilpotents de dimension$\leq6.$}, PhD thesis, Lorraine University $2014$.
\bibitem{watson} G. N. Watson, \textit{A Treatise on the Theory of Bessel Functions}, 2nd  edition., 1995, Cambridge University Press.
\end{thebibliography}
\end{document}